\documentclass{amsart}

\usepackage{pifont}
\usepackage{amssymb}
\usepackage{txfonts}
\usepackage{cleveref}
\usepackage{microtype}
\usepackage{mathrsfs}
\usepackage{graphicx}
\usepackage{cite}
\usepackage{xcolor}

\newtheorem{theorem}{Theorem}[section]
\newtheorem{lemma}[theorem]{Lemma}
\newtheorem{corollary}[theorem]{Corollary}

\theoremstyle{definition}
\newtheorem{definition}[theorem]{Definition}
\newtheorem{example}[theorem]{Example}
\newtheorem{proposition}[theorem]{Proposition}

\theoremstyle{remark}
\newtheorem{remark}[theorem]{Remark}

\numberwithin{equation}{section}

\begin{document}

\title{Topological expansive Lorenz maps with a hole at critical point}

\author[Y.\ Sun]{Yun Sun}
\address[Yun Sun]{Department of Mathematics, South China University of Technology, Guangzhou, 510461, China}
\email{masy2021@mail.scut.edu.cn}
\author[B.\ Li]{Bing Li}
\address[Bing Li]{Department of Mathematics, South China University of Technology, Guangzhou, 510461, China}
\email{scbingli@scut.edu.cn}
\author[Y M.\ Ding]{Yiming Ding}

\address[Yiming Ding]{College of Science, Wuhan University of Science and Technology, Wuhan, 430065, China}
\email{dingym@wust.edu.cn}

\begin{abstract}
Let $f$ be an expansive Lorenz map and $c$ be the critical point. The survivor set is denoted as $S_{f}(H):=\{x\in[0,1]: f^{n}(x)\notin H, \forall n\geq 0\}$, where $H$ is a open subinterval. Here we study the hole $H=(a,b)$ with $a\leq c \leq b$ and $a\neq b $. We show that the case $a=c$ is equvalent to the hole at $0$, the case $b=c$ equals to the hole at $1$. We also obtain that, given an expansive Lorenz map $f$ with a hole $H=(a,b)$ and $S_{f}(H)\nsubseteqq\{0,1\}$, then there exists a Lorenz map $g$ such that $\tilde{S}_{f}(H)\setminus\Omega(g)$ is countable, where $\Omega(g)$ is the Lorenz-shift of $g$ and $\tilde{S}_{f}(H)$ is the symbolic representation of $S_{f}(H)$. Let $a$ be fixed and $b$ varies in $(c,1)$, we also give a complete characterization of the maximal interval $I(b)$ such that for all $\epsilon\in I(b)$, $S_{f}(a,\epsilon)=S_{f}(a,b)$, and $I(b)$ may degenerate to a single point $b$. Moreover, when $f$ has an ergodic acim, we show that the topological entropy function $\lambda_{f}(a):b\mapsto h_{top}(f|S_{f}(a,b))$ is a devil staircase with $a$ being fixed, so is $\lambda_{f}(b)$ if we fix $b$. At the special case $f$ being intermediate $\beta$-transformation, using the Ledrappier-Young formula, we obtain that the Hausdorff dimension function $\eta_{f}(a):b\mapsto \dim_{\mathcal{H}}(S_{f}(a,b))$ is a devil staircase when fixing $a$, so is $\eta_{f}(b)$ if $b$ is fixed. As a result, we extend the devil staircases in \cite{Urbanski1986,kalle2020,Langeveld2023} to expansive Lorenz maps with a hole at critical point.

\end{abstract}

\maketitle

\section{Introduction}\label{intro1}
The study of dynamical systems with holes, i.e. the characterization of points which do not
fall into certain predetermined sets under iteration by a map, are called open dynamical systems, and were first proposed by Pianigiani and Yorke \cite{Pianigiani1979} in 1979. In recent years open dynamical systems have posed interesting questions both about arithmetic properties of points and their dynamical interpretation (cf.\cite{Sidorov2003,Sidorov2014,glendinning2015,Agarwal2020}). In the general setting, one considers a discrete dynamical system $(X,T)$,  where $X=[0,1]$ and $T: X\rightarrow X$ is continuous with positive topological entropy. Let $H\subset X$ be a connected subinterval, called the hole. We focus on the following survivor set corresponding to hole $H$:
$$ S_{T}(H)=\{x\in X: T^{n}(x)\notin H \  \forall n\geq 0\}=X\setminus\bigcup_{n=0}^{\infty}T^{-n}(H).
$$
It can be seen that the size of $S_{T}(H)$ depends not only on the size of $H$ but also on the position of the hole.
\par There are many works \cite{Agarwal2020,kalle2020,Langeveld2023,Urbanski1986,glendinning2015,clark2016} concerning about the Hausdorff dimension of the survivor set $S_{T}(H)$ with $T$ being piecewise linear transformations. The classic result by Urba${\rm\acute{n}}$ski \cite{Urbanski1986,Urbanski1987} mainly considered $C^{2}$-expanding, orientation-preserving circle maps with a hole $(0,t)$. He proved that for the doubling map $T_{2}$, the Hausdorff dimension of the survivor set $S_{2}(t):=\{x\in [0,1): T_{2}^{n}(x)\notin (0,t) \  \forall n\geq 0\}$ depends continuously on the parameter $t\in[0,1)$. Furthermore, he showed that the dimension function $\eta_{2}:t\mapsto \dim_{\mathcal{H}}S_{2}(t)$ is a devil staircase, and studied its bifurcation set. For the doubling map with an arbitrary hole $(a,b)\subset[0,1)$, Glendinning and Sidorov \cite{glendinning2015} studied the survivor set $S_{2}(a,b):=\{x\in [0,1): T_{2}^{n}(x)\notin (a,b) \  \forall n\geq 0\}$, and determined when $S_{2}(a,b)$ is nonepmty, infinite, or has positive Hausdorff dimension. They proved that when the size of the hole $(a,b)$ is strictly smaller than 0.175092, then $\dim_{\mathcal{H}}(S_{2}(a,b))>0$. Later, Clark \cite{clark2016} partially extended Glendinning and Sidorov's result to the $\beta$-dynamical system $([0,1), T_{\beta})$, where $\beta\in(1,2]$ and $T_{\beta}(x):=\beta x \ ( \bmod \ 1)$.

\begin{figure}[t]
\begin{center}
\includegraphics[width=1.0\textwidth]{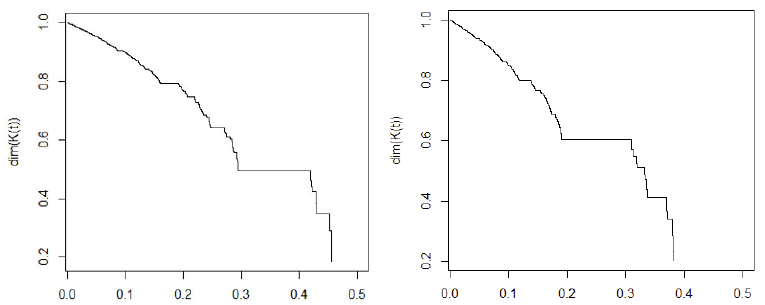}

\caption{Left: $\eta_{\beta}:t\mapsto \dim_{\mathcal{H}}S_{\beta}(t)$ where $\beta$ is tribonacci number. \ Right: $\eta_{\beta,\alpha}:t\mapsto \dim_{\mathcal{H}}S_{\beta,\alpha}(t)$ where $\beta$ is golden mean and $\alpha=1-\beta/2$.}
\label{fig1}
\end{center}
\end{figure}

\par Motivated by the works above, Kalle et al.\cite{kalle2020} considered the survivor set in the $\beta$-dynamical system with a hole at zero. Similar to Urba${\rm\acute{n}}$ski's result on doubling maps, they also determined the Hausdorff dimension of the survivor set $$S_{\beta}(t):=\{x\in [0,1): T_{\beta}^{n}(x)\notin (0,t) \  \forall n\geq 0\},$$
and showed that the dimension function $\eta_{\beta}:t\mapsto \dim_{\mathcal{H}}S_{\beta}(t)$ is a non-increasing devil staircase, see the left graph in Figure \ref{fig1}. Moreover, in another paper by Allaart and Kong \cite{allaart2022}, they gave a characterization of the critical value
$\tau(\beta)=\sup\{t:\dim_{\mathcal{H}}S_{\beta}(t)>0\} =\inf\{t:\dim_{\mathcal{H}}S_{\beta}(t)=0\}$
for each $\beta\in(1,2]$.
 \par Based on Kalle et al.'s work, Langeveld and Samuel \cite{Langeveld2023} recently consider the intermediate $\beta$-shifts with a hole at zero. They focus on intermediate $\beta$-transformation $T_{\beta,\alpha}: x\mapsto \beta x+\alpha \ ( \bmod \ 1)$ acting on $[0,1]$, where $(\beta,\alpha)\in\Delta:=\{(\beta, \alpha) \in \mathbb{R}^{2}:\beta \in (1, 2) \; \rm{and} \; \alpha \in[0,2 - \beta]\}.$  The survivor set is defined as
 $$S_{\beta,\alpha}(t):=\{x\in [0,1]: T_{\beta,\alpha}^{n}(x)\notin (0,t) \  \forall n\geq 0\}.
 $$
 They show that every intermediate $\beta$-transformation is topologically conjugate to a greedy $\beta$-transformation with a hole at zero, and provide a counterexample illustrating that the correspondence is not one-to-one. As an application of this observation, they obtain that the dimension function $\eta_{\beta,\alpha}:t\mapsto \dim_{\mathcal{H}}(S_{\beta,\alpha}(t))$ is a non-increasing devil staircase function, extending the work of $S_{\beta}(t)$ by Kalle et al., see the right graph in Figure \ref{fig1}.
 Note, through this article, by topologically conjugate we mean that
the conjugacy is one-to-one everywhere except on a countable set on which the conjugacy is at most finite to one.

\par Observed that the devil staircase results mentioned above mainly refer to $\beta$-transformations or intermediate $\beta$-transformations with a hole at zero. Two natural questions arise,
\begin{enumerate}
\item Can we extend the intermediate $\beta$-transformation to more general cases: expansive Lorenz maps?
\item If considering the hole at critical point $c$, can we still obtain the devil staircase?
\end{enumerate}
 In this paper, we give positive answers to the two questions. Here we focus on expansive Lorenz maps. A {\em Lorenz map} on $X=[0,1]$ is an interval map $f:X \to X$ such
that for some
$c\in (0,1)$ we have\\
\indent
(i)   $f$ is strictly increasing on $[0,c)$ and on $(c,1]$;\\
\indent (ii)  $\lim_{x \uparrow c}f(x)=1$, $\lim_{x \downarrow
	c}f(x)=0$.

If, in addition, $f$ satisfies the topological expansive condition

(iii) The preimages set $C(f)=\cup_{n \ge 0}f^{-n}(c)$ of $c$ is dense
in $X$,
then $f$ is said to be an {\bf expansive Lorenz map}. Lorenz maps are one-dimensional maps with a single discontinuity, which
arise as Poincar${\rm \acute{e}}$ return maps for flows on branched manifolds that model the strange
attractors of Lorenz systems. There are lots of studies about Lorenz maps\cite{cuiding2015,ding2011,glendinning1993,sun2023}, such as renormalization\cite{glendinning1996,hubbard1990,glendinning1993} , kneading invariants\cite{cuiding2015,ds2021,glendinning1990} and so on. For convenience, we denote $ELM$ be the set of expansive Lorenz maps, and $LM$ be the set of Lorenz maps.
\par Denote $\tilde{S}_{f}(H)$ as the symbolic representation of the survivor set $S_{f}(H)$, we wonder if there are some connections between the hole at zero and the hole at critical point? Let us see the following definitions of sets $S^{+}_{f}(H)$, and $\tilde{S}^{+}_{f}(H)$ is the symbolic representation.

$$
S^{+}_{f}(H)=\left \{
\begin{array}{ll}
\{x\in[0,1]: f^{n}(x)\geq f(b) \ \forall n\geq0\} & H=(0,f(b)), \\
\{x\in[0,1]: f^{n}(x)\leq f(a) \ \forall n\geq0\} & H=(f(a),1), \\
\{x\in[0,1]:f(b)\leq f^{n}(x)\leq a \ {\rm or} \ b \leq f^{n}(x) \leq f(a)\ \forall n\geq0\} & H=(a,b).
\end{array}
\right.
$$
Notice that the hole $H=(a,b)$ studied in this paper always satisfying $a\leq c\leq b$ and $a\neq b$.
\begin{theorem}\label{xinthm1}
Let $f\in ELM$ with a hole $(a,b)$. If $a=c$, then $S^{+}_{f}(c,b)=S^{+}_{f}(0,f(b))$, and $S_{f}(c,b)=S_{f}(0,f(b))$ if and only if $0\in S_{f}(0,f(b))$.
\end{theorem}

\begin{remark}
\
\par
\begin{enumerate}
\item If $b=c$, then $S^{+}_{f}(a,c)=S^{+}_{f}(f(a),1)$, and $S_{f}(a,c)=S_{f}(f(a),1)$ iff $1\in S_{f}(f(a),1)$.
\item If $0\notin S_{f}(0,f(b))$, then $S_{f}(c,b)\setminus S_{f}(0,f(b))$ is countable. So is it for the set $S_{f}(a,c)\setminus S_{f}(f(a),1)$ if $1\notin S_{f}(f(a),1)$.
\item By Proposition \ref{pro1}, $h_{top}(S_{f}(c,b))=h_{top}(S_{f}(0,f(b)))$. When $f$ being $T_{\beta}$ or $T_{\beta,\alpha}$ ($\beta\in(1,2)$ and $\alpha\in(0,2-\beta)$), all the results about the hole at zero \cite{kalle2020,Langeveld2023} can be naturally extended to the hole $(c,b)$ where $b$ varies in $(c,1)$.
\end{enumerate}
\end{remark}

\par  As a result, the case $a=c$ equals to the hole at 0, and the case $b=c$ is equivalent to the hole at $1$.
\begin{theorem}\label{theorem1}
Given $f\in ELM$ with a  hole $H=(a,b)$ and $\tilde{S}^{+}_{f}(H)\nsubseteqq\{0^\infty,1^\infty\}$, there exists $g\in LM$ such that $\Omega(g)=\tilde{S}^{+}_{f}(H)$. Conversely, given $g\in LM$, there exist $f\in ELM$ with a hole $H=(a,b)$ such that $\Omega(g)= \tilde{S}^{+}_{f}(H)$.
\end{theorem}

\par As we can see, Theorem \ref{theorem1} obtains the essential connection between $g\in LM$ and open dynamical systems driven by $f\in ELM$. In Proposition \ref{pro1}, we prove that $h_{top}(\tilde{S}^{+}_{f}(H))=h_{top}(\tilde{S}_{f}(H))$, hence we are able to calculate the topological entropy of survivor set via the kneading determinants of Lorenz map $g$. Moreover, similar to the results on intermediate $\beta$-transformations \cite{Langeveld2023}, we observe that every expansive Lorenz map with a hole $H=(a,b)$ is topologically conjugate to another Lorenz map.
Next we consider the following bifurcation set $E_{f}(a)$ with $a$ being fixed,
$$ E_{f}(a):=\{b\in[c,1]:S^{+}_{f}(a,\epsilon)\neq S^{+}_{f}(a,b) \ {\rm for \ any} \ \epsilon>b, {\rm \ where} \ a\in[0,c]\ {\rm is \ fixed}\}.
$$
In particular, if $a=c$, we write the bifurcation set $E_{f}(a)$ as $E_{f}$. Since it would be similar to consider the case $b$ being fixed, here we only prove the case $E_{f}(a)$.

 \par There are many metric and topological results on the bifurcation sets before, and mainly focus on $E_{f}$, i.e., $a=c$ and is equivalent to the hole at zero. Urb${\rm\acute{a}}$nski \cite{Urbanski1986} proved that $E_{T_{2}}$ is a Lebesgue null set of full Hausdorff dimension when $f$ is the doubling map. Kalle et al. \cite{kalle2020} extended this to $T_{\beta}$ with $\beta\in(1,2]$, and  obtained that $E_{T_{\beta}}$ is of null Lebesgue measure with full Hausdorff dimension, they also considered the topological structure of the bifurcation sets. When $f$ being $T_{\beta,\alpha}$ with $(\beta,\alpha)\in\Delta$, Langeveld and Samuel \cite{Langeveld2023} proved that $E_{T_{\beta,\alpha}}$ is of null Lebesgue measure but with full Hausdorff dimension. Observed that all the proofs of Lebesgue null set require $f$ owning ergodic and absolutely continuous invariant measure w.r.t. Lebesgue measure, but the existence of ergodic acim for general expansive Lorenz maps is unknown. However, Ding, Fan and Yu \cite{DFY2010} gave the sufficient and necessary conditions for a piecewise linear Lorenz map owning ergodic acim. As a result, the following theorem also holds for some piecewise linear Lorenz maps.

\begin{theorem}\label{theorem2}
Let $f\in ELM$ with ergodic a.c.i.m.. Then the topological entropy function $\lambda_{f}(a):b\mapsto h_{top}(\tilde{S}_{f}(H))$ is a devil staircase, where $H=(a,b)$ and $a$ is fixed.
\end{theorem}
It is clear to see that $\lambda_{f}(a)$ is decreasing. However, if we consider the case $b$ being fixed, the entropy function $\lambda_{f}(b)$ is also a devil staircase, the only difference is that $\lambda_{f}(b)$ is increasing. When we consider $f$ being an intermediate $\beta$-transformation, its Lyapunov exponent is $\log \beta$, with the help of Ledrappier-Young formula due to Raith \cite{raith1994}, we know that the Hausodrff dimension
$$ \dim_{\mathcal{H}}(S_{T_{\beta,\alpha}}(H))=\frac{h_{top}(T_{\beta,\alpha}|
S_{T_{\beta,\alpha}}(H))}{\log \beta}.
$$
As an application of Theorem \ref{theorem2} and the dimension formula, we can quickly obtain the following corollary.

\begin{corollary}
Let $f=T_{\beta,\alpha}$ with hole $H=(a,b)$, where $(\beta,\alpha)\in\Delta$ and $a\in [0,\frac{1-\alpha}{\beta}]$ is fixed. Then the Hausdorff dimension function $\eta_{f}(a):b\mapsto \dim_{\mathcal{H}}(S_{f}(H))$ is a devil staircase, that is, $\eta_{f}(a)$ is decreasing, and $\eta_{f}(a)$ is constant Lebesgue almost everywhere.
\end{corollary}

\par Our work is organized as follows. In Section 2, we introduce some preliminary concepts, including kneading invariants, H-S admissible condition and kneading determinants. Section 3 gives some valuable lemmas and presents the proof of Theorem \ref{xinthm1} and \ref{theorem1}. We study the plateaux of survivor sets $\tilde{S}^{+}_{f}(a,b)$ in Section \ref{platformsec} and prove Theorem \ref{theorem2}. Finally, Section \ref{comments} concludes the paper with some additional comments for future research.
To aid in visualizing our results, we include various interesting examples in the appendix.

\section{Preliminaries}
\subsection{Notations}
\
\par We equip the space $\{0,1\}^\mathbb{N}$ of infinite sequences with the topology induced by the usual metric $d \colon \{0,1\}^\mathbb{N} \times \{0,1\}^\mathbb{N} \to \mathbb{R}$ which is given by
\begin{align*}
d(\omega, \nu) \coloneqq
\begin{cases}
0 & \text{if} \; \omega = \nu,\\
2^{- \lvert\omega \wedge \nu\rvert + 1} & \text{otherwise}.
\end{cases}
\end{align*}
Here $\rvert \omega \wedge \nu \lvert \coloneqq \min \, \{ \, n \in \mathbb{N} \colon \omega_{n} \neq \nu_n \}$, for all $\omega = (\omega_{1}\omega_{2}\dots) , \nu = ( \nu_{1} \nu_{2}\dots) \in \{0, 1\}^{\mathbb{N}}$ with $\omega \neq \nu$. Note that the topology induced by $d$ on $\{ 0, 1\}^{\mathbb{N}}$ coincides with the product topology on $\{ 0, 1\}^{\mathbb{N}}$.  For $n\in\mathbb{N}$ and $\omega\in\{0,1\}^{\mathbb{N}}$, we set $\omega|_{1}^{n}=\omega|_{n}=(w_{1}\cdots w_{n})$ and call $n$ the length of $\omega|_{n}$. We let $\sigma \colon \{ 0, 1 \}^{\mathbb{N}} \circlearrowleft$ denote the \textsl{left-shift map} which is defined by $\sigma(\omega_{1} \omega_{2} \dots) \coloneqq (\omega_{2} \omega_{3}\dots)$.  A \textsl{subshift} is any closed subset $\Omega \subseteq \{0,1\}^\mathbb{N}$ such that $\sigma(\Omega) \subseteq \Omega$.  Given a subshift $\Omega$ and $n \in \mathbb{N}$ we set
\begin{align*}
\Omega\lvert_{n} \coloneqq \left\{ (\omega_{1} \dots \omega_{n}) \in \{ 0, 1\}^{n} \colon \text{there exists} \, \omega \in \Omega \, \ \text{with} \ \, \omega|_{n} = (\omega_{1} \dots \omega_{n}) \right\}
\end{align*}
and denote by $\Omega^{*} \coloneqq \bigcup_{n = 1}^{\infty} \Omega\lvert_{n}$ for the collection of all finite words. For $\xi\in\Omega^{*}$, we denote $|\xi|$ as the length of $\xi$. And we denote by $\#\Omega|_{n}$ the cardinality of $\Omega|_{n}$.
\par For $n, m \in \mathbb{N}$ and $\nu = (\nu_{1}\dots \nu_{n}),\, \xi = (\xi_{1} \dots \xi_{m}) \in \{ 0, 1\}^{*}$, set
$
\nu \xi\coloneqq (\nu_{1} \dots \nu_{n} \xi_{1} \dots\xi_{m});
$
we use the same notation when $\xi \in \{ 0, 1\}^{\mathbb{N}}$.  An infinite word $\omega = (\omega_{1} \omega_{2} \dots) \in \{0, 1\}^{\mathbb{N}}$ is called \textsl{periodic} with \textsl{period} $n \in \mathbb{N}$ if and only if, $(\omega_{1}\dots \omega_{n}) = (\omega_{(m - 1)n + 1} \dots \omega_{m n})$, for all $m \in \mathbb{N}$; in which case we write $\omega = (\omega_{1} \dots \omega_{n})^{\infty}$.  Similarly,  $\omega = (\omega_{1}\omega_{2} \dots) \in \{0, 1\}^{\mathbb{N}}$ is called \textsl{eventually periodic} with \textsl{period} $n \in \mathbb{N}$ if and only if there exists $k \in \mathbb{N}$ such that $(\omega_{k+1} \dots \omega_{k+n}) = (\omega_{k+(m - 1)n + 1} \dots, \omega_{k+ m n})$ for all $m \in \mathbb{N}$; in which case we write $\omega = \omega_{1}\dots \omega_{k} (\omega_{k+1} \dots \omega_{k+n})^{\infty}$.

\subsection{Kneading invariants and H-S admissible condition}
\
\par Let $f$ be an expansive Lorenz map. The trajectories of points in $[0,1]$ by $f$ can be coded by elements of $\{0,\ 1\}^{\mathbb{N}}$, which denotes the set of infinite sequences on the alphabet $\{0,\ 1\}$. Denote $\omega = (\omega_{1}\omega_{2}\dots) , \nu = ( \nu_{1} \nu_{2}\dots) \in \{0,\ 1\}^{\mathbb{N}}$
, the lexicographic order on $\{0,\ 1\}^{\mathbb{N}}$ is the total order defined $\omega\prec \nu$ if $\omega\neq \nu$ and $\omega_n<\nu_n$, where $n$ is the least index such that $\omega_n\neq \nu_n$.

 The kneading
sequence of a point $x$, $\tau_f(x)$, is defined to be $\epsilon_0
\epsilon_1 \ldots$ of $0's$ and $1's$ as follows:
$$ \epsilon_i=0\ \ \ \ \ \ {\rm if}\ \ \ f^i(x)<c \ \ \ \ {\rm and} \ \ \
\epsilon_i=1\ \ \ \ \ \ {\rm if}\ \ \ f^i(x)>c.$$
This definition works for $x \notin C(f)=\cup_{n \ge 0}f^{-n}(c)$. In the case where $x$ is
a preimage of $c$, $x$ has upper and lower kneading sequences
$$ \tau_{f}(x+)=\lim_{y\downarrow x}\tau_f(y), \ \ \ \ \ \ \ \ \
\tau_{f}(x-)=\lim_{y\uparrow x}\tau_f(y),$$ where the $y's$ run through
points of $[0,1]$ which are not the preimages of $c$. In particular,
$(k_+, k_-)=(\tau_{f}(c+), \tau_{f}(c-)) $ are called the {\bf kneading
invariants} of $f$, which were used to developing the renormalization
theory of expansive Lorenz map by Glendinning and Sparrow \cite{hubbard1990}.
The kneading space of $f$, also called Lorenz shift, is
$$\Omega(f)=\{\tau_f(x): x \in I\}.$$
Since $\sigma$ is the shift map operating on the Lorenz shift
$\Omega(f)$, then clearly $\tau_{f}(f(x))=\sigma(\tau_f(x))$, with similar
results holding for the upper and lower kneading sequences of points
$x$ which are preimages of $c$. Moreover, we denote $k(0)=\sigma(k_{+})$ and $k(1)=\sigma(k_{-})$. The dynamics of $f$ on
$I$ can be modeled by the shift map $\sigma$ on kneading space.

\begin{theorem}{\bf({\cite[Theorem 2]{hubbard1990}})}\label{rem:k+k-}
Let $f$ be a Lorenz map, the kneading space $\Omega(f)$ is completely determined by the kneading invariants $(k_+,k_-)$ of $f$;  indeed, we have
\begin{align*}
\ \ \ \ \ \ \ \ \Omega(f)&= \left\{ \omega \in \{ 0, 1\}^{\mathbb{N}} \colon k(0) \preceq \sigma^{n}(\omega) \preceq k_{-} \, \ \textup{or} \ \, k_{+} \preceq \sigma^{n}(\omega) \preceq k(1) \, \ \textup{for all} \, n \in \mathbb{N}_{0} \right\}\!,\\
 &= \left\{ \omega \in \{ 0, 1\}^{\mathbb{N}} \colon k(0) \preceq  \sigma^{n}(\omega) \preceq k(1) \ \textup{for all} \, n \in \mathbb{N}_{0} \right\}\!.
\end{align*}
Moreover,
$
\Omega(f)$ is closed with respect to the metric $d$ and hence is a subshift.
\end{theorem}

We can see that each $f\in ELM$ corresponds to a Lorenz shift $\Omega(f)$, and $\Omega(f)$ is totally determined by its kneading invariants $(k_{+},k_{-})$. If we know the kneading invariants of $f$, we also write the Lorenz shift as $\Omega(k_{+},k_{-})$. A natural question stemming from above is that, what kind of sequences in $\{ 0, 1 \}^{\mathbb{N}}$ can be the kneading invariants of an expansive Lorenz map? It was solved by Hubbard and Sparrow in \cite{hubbard1990} and we usually call it {\bf H-S admissible condition}, which stated as follows.

\begin{theorem}{\bf({\cite[Theorem 1]{hubbard1990}})} \label{kneading space}
If $f\in ELM$, then its kneading invariants $(k_{+},k_{-})$ satisfy
	\begin{equation} \label{expanding}
		\sigma (k_{+})\preceq \sigma^n(k_{+})\prec\sigma (k_{-}), \ \ \ \ \ \ \sigma (k_{+})\prec \sigma^n(k_{-})\preceq\sigma (k_{-}) \ \ \ \  \forall n \ge 0,
	\end{equation}
 Conversely, given any two sequences $k_{+}$ and $k_{-}$ satisfying $(\ref{expanding})$, there exists $f\in ELM$ with $(k_{+},k_{-})$ as its kneading invariant, and $f$ is unique up to conjugacy.
\end{theorem}

\par Theorem \ref{kneading space} gives the admissible condition for the kneading invariants of an expansive Lorenz map. Two expansive Lorenz maps are said to be
 equivalent, if they admit the same kneading invariants. When we do not want to mention $f$, we also write
$\Omega_{K}$ as the kneading space, where $K=(k_{+},k_-)$ be the kneading invariants. In particular, we give the definition of {\bf weak-admissible}, which corresponds to the non-expansive cases. We say the kneading invariants $(k_{+},k_{-})$ are weak-admissible if satisfying
	\begin{equation} \label{non-expanding}
\sigma(k_{+}) \preceq\sigma^{n}(k_{+})\preceq\sigma(k_{-}) \ {\rm and} \ \sigma(k_{+})\preceq\sigma^{n}(k_{-})\preceq\sigma(k_{-}) \ \ \ \ {\rm for  \ all}\ \ n\geq 0,
	\end{equation}
which means there may exist $n$ such that $\sigma^{n}(k_{+})=k_{-}$ or $\sigma^{n}(k_{-})=k_{+}$. Clearly, H-S admissibility implies weak admissibility, but not vice versa. There are many trivial examples about weak-admissibility, such as the kneading invariants induced by rational rotations.

\subsection{Kneading determinant and topological entropy}\label{determinant}
\
\par The ideas for kneading determinant goes back to \cite{milnor1988}, see also \cite{glendinning1996}. Let $({k_ + },{k_ - })$ be the kneading invariants of $f\in LM$, where $k_{+}=(v_{1}v_{2}\cdots)$ and $k_{-}=(w_{1}w_{2}\cdots)$. Then the kneading determinant is a formal power series defined as $K(t) = {K_ + }(t) - {K_ - }(t)$ , where
$${K_ + }(t) = \sum\limits_{i = 1}^\infty  {v_{i}{t^{i-1}}},\ \ \ \ {K_ - }(t) = \sum\limits_{i = 1}^\infty  {w_{i}{t^{i-1}}}.$$
 The following lemma offers a direct way to calculate topological entropy of $\Omega(f)$ if the kneading invariants are given.
\begin{lemma}{\bf({ \cite[Theorem 3]{glendinning1996}\cite[Lemma 3]{barnsley2014}})}  \label{zeros}
Let $({k_ + },{k_ - })$ be the kneading invariants of $f\in LM$ with $h_{top}(f)>0$, and $K(t)$ be the corresponding kneading determinant. Denote $t_{0}$ be the smallest positive root of $K(t)$ in $(0,1)$, then $h_{top}(f)=-\log t_{0}$.
\end{lemma}

\begin{remark}\label{t0beta}
If $(k_{+},k_{-})$ corresponds to some intermediate $\beta$-shift, which means $\Omega(k+,k_-)=\Omega_{\beta,\alpha}$ for some $(\beta,\alpha)\in\Delta$, then $1/\beta$ equals the smallest positive root of $K(t)$ in $(0,1)$.

\end{remark}

\par The papers by Raith \cite{raith1989,raith1992,raith1994} studied invariant sets for piecewise monotone expanding maps on the interval $[0,1]$. In particular, Raith \cite{raith1994} removed a finite number of open intervals from $[0,1]$ and considered piecewise monotone expanding maps restricted to the survivor set. He also studied the dependence on the end points of the hole of the topological entropy of the map restricted to the survivor set. As a result, we can apply these results to $f\in ELM$ on $[0,1]$ with the single hole $H$ removed. Moreover, applying the results from {\cite[Corollary 1.1 and Theorem 2]{raith1994}} gives the following.

\begin{proposition}  {\bf({\cite[Lemma 8]{raith1994}})}   \label{continuity}
Let $f\in ELM$ with a hole $H=(a,b)$. The entropy map $\lambda:b\mapsto h_{top}(S_{f}(H))$ is continuous on $[0,1]$, where $a$ is fixed.

\end{proposition}

\section{Proof of Theorem 1.1 \& Theorem 1.2}
Before stating the proof, we give some useful lemmas. Here we always let $a\leq c \leq b$ when considering the hole $(a,b)$, and denote ${\rm \bf {a}}=\tau_{f}(a-)$ as the lower kneading sequence of $a$ and ${\rm \bf {b}}=\tau_{f}(b+)$ be the upper kneading sequence of $b$. In order to show the connection between the hole at zero and the hole at critical point, we prove Lemma \ref{lemma2} with three cases. Recall that $\tilde{S}_{f}(H)$ is the symbolic representation of the survivor set $S_{f}(H)$, and the definition of $S^{+}_{f}(H)$ refers to Section \ref{intro1}.

\begin{lemma}\label{lemma2}
Let $f\in ELM$ with a hole $H$. The set $S_{f}(H)\setminus S^{+}_{f}(H)$ is at most countable, where the hole can be the following three cases,
$$
H=\left \{
\begin{array}{ll}
(a,b) \ \ {\rm with} \ \ a\leq c \leq b, \\
(0,t) \ \ {\rm with} \ \  t<c, \\
(s,1)\ \ {\rm with} \ \  s>c.
\end{array}
\right.
$$

\begin{proof}
 Denote $(k_{+},k_{-})$ as the kneading invariants of $f$ and $k(0)=\sigma(k_{+})$, $k(1)=\sigma(k_{-})$. We divide the proof into three cases. Case 1, $H=(a,b)$ with $a<c<b$. Case 2, $H=(a,b)$ with $a=c$ or $b=c$. Case 3, $H=(0,t)$ with $t<c$. The case $H=(s,1)$ with $s>c$ is essentially the same with case 3, which can be obtained similarly. We prove this lemma via symbolic representations of sets $S^{+}_{f}(H)$ and $S_{f}(H)$.
\par Case 1, $a<c<b$. If ${\rm \bf {a}}|_{2}=00$ or ${\rm \bf {b}}|_{2}=11$,
 it is clear that $S_{f}(a,b)\subseteq\{0,1\}$, and $S_{f}(a,b)=\{0,1\}$ if and only if both $0$ and $1$ are the fixed points of $f$. Hence we only need to consider the case that ${\rm \bf {a}}|_{2}=01$ and ${\rm \bf {b}}|_{2}=10$. By the definition of survivor sets, we write the symbolic forms of $S_{f}(a,b)$ as
$$
\tilde{S}_{f}(a,b)=\{w\in\{0,1\}^{\mathbb{N}}: \sigma^{n}(w)\preceq {\rm \bf {a}} \ \textup{or} \ \sigma^{n}(w)\succeq {\rm \bf {b}} \ \forall n \in \mathbb{N}_{0}\},$$
and denote
\begin{align*}
\ \ \ \ \ \ \ \ \tilde{S}^{+}_{f}(a,b)&=\{w\in\{0,1\}^{\mathbb{N}}: \sigma({\rm \bf {b}})\preceq\sigma^{n}(w)\preceq {\rm \bf {a}} \ \textup{or} \ {\rm \bf {b}}\preceq\sigma^{n}(w)\preceq \sigma({\rm \bf {a}}) \ \forall n \in \mathbb{N}_{0}\}.\!,\\
 &= \left\{ w \in \{ 0, 1\}^{\mathbb{N}} : \sigma({\rm \bf {b}}) \preceq  \sigma^{n}(\omega) \preceq \sigma({\rm \bf {a}})\ \forall \, n \in \mathbb{N}_{0} \right\}\!.
\end{align*}
The second equality of $\tilde{S}^{+}_{f}(a,b)$ can be obtained easily as follows. Denote $A=\{w\in\{0,1\}^{\mathbb{N}}: \sigma({\rm \bf {b}})\preceq\sigma^{n}(w)\preceq {\rm \bf {a}} \ \textup{or} \ {\rm \bf {b}}\preceq\sigma^{n}(w)\preceq \sigma({\rm \bf {a}}) \ \forall n \in \mathbb{N}_{0}\}$ and $B=\{ w \in \{ 0, 1\}^{\mathbb{N}} : \sigma({\rm \bf {b}}) \preceq  \sigma^{n}(\omega) \preceq \sigma({\rm \bf {a}})\ \forall \, n \in \mathbb{N}_{0} \}$. It is clear that $A\subseteq B$. Denote $w=w_{1}\cdots w_{n}\cdots$, and suppose $w\notin A$, then there exists $n$ such that ${\rm \bf {a}}\prec\sigma^{n}(w)\prec{\rm \bf {b}}$. If $w_{n+1}=0$, then $\sigma^{n}(w)\succ{\rm \bf {a}}$ implies $\sigma^{n+1}(w)\succ\sigma({\rm \bf {a}})$ and thus $w\notin B$. Similarly, if $w_{n+1}=1$, then $\sigma^{n}(w)\prec{\rm \bf {b}}$ implies $\sigma^{n+1}(w)\prec\sigma({\rm \bf {b}})$ and hence $w\notin B$.

\par Indeed, we can see that for any $w\in\tilde{S}_{f}(a,b)$ begin with $0$ is lexicographically smaller than $ {\rm \bf {a}}$, which indicates that all the sequences begin with $1$ will lexicographically smaller than $\sigma( {\rm \bf {a}})$, excluding the sequence $1^{\infty}$. Similarly we have that each sequence begin with $0$ is lexicographically larger than $\sigma( {\rm \bf {b}})$, excluding the sequence $0^{\infty}$. As a result, the set $\tilde{S}_{f}(a,b)\setminus\tilde{S}^{+}_{f}(a,b)\subseteq\{0^{\infty},1^{\infty}\}$, which is countable.

\par Case 2, $a=c$ or $b=c$. Here we only prove the case $a=c$, and the case $b=c$ can be obtained similarly. Denote $\tilde{S}_{f}(c,b)$ as the symbolic form of survivor set $S_{f}(c,b)$. Actually, due to the specificity of point $c$, $\tilde{S}_{f}(c,b)$ can be written as the union of two disjoint sets $\tilde{S}^{1}_{f}(c,b)$ and $\tilde{S}^{2}_{f}(c,b)$, where
\begin{align*}
\ \ \ \ \ \ \ \ \tilde{S}^{1}_{f}(c,b)&= \{w\in\{0,1\}^{\mathbb{N}}: \textup{there \ exists \ sequence} \ \{n_{i}\}_{i\geq1} \ \textup{ such \ that} \ \sigma^{n_{i}}(w)=k_+,\!\\
 &   \textup{and} \ \sigma^{n}(w)\preceq k_{-} \ \textup{or} \ \sigma^{n}(w)\succeq{\rm \bf {b}} \ \textup{for \ all \ other } \ n\geq0  \},\!
\end{align*}
 and
$$
\tilde{S}^{2}_{f}(c,b)=\{w\in\{0,1\}^{\mathbb{N}}: \sigma^{n}(w)\preceq k_{-} \ \textup{or} \ \sigma^{n}(w)\succeq {\rm \bf {b}} \ \forall n \in \mathbb{N}_{0}\}.
$$
As we can see, $\tilde{S}^{1}_{f}(c,b)$ is countable since it is a subset of $k_+$'s preimages. Moreover, similar to Case 1 above, it can be checked that
$$ \tilde{S}^{2}_{f}(c,b)\setminus\tilde{S}^{+}_{f}(c,b)\subseteq\{0^{\infty},1^{\infty}\},
$$
where
\begin{align*}
\ \ \ \ \ \ \ \ \tilde{S}^{+}_{f}(c,b)&=\{w\in\{0,1\}^{\mathbb{N}}: \sigma({\rm \bf {b}})\preceq\sigma^{n}(w)\preceq k_- \ \textup{or} \ {\rm \bf {b}}\preceq\sigma^{n}(w)\preceq k(1) \ \forall n \in \mathbb{N}_{0}\}\!\\
 &= \left\{ w \in \{ 0, 1\}^{\mathbb{N}} : \sigma({\rm \bf {b}}) \preceq  \sigma^{n}(\omega) \preceq k(1)\ \forall \, n \in \mathbb{N}_{0} \right\}.\!
\end{align*}
And $\{0^{\infty},1^{\infty}\}$ can be reached if and only if both $0$ and $1$ are the fixed points of $f$. As a result, $\tilde{S}_{f}(c,b)\setminus\tilde{S}^{+}_{f}(c,b)\subset\big(\tilde{S}^{1}_{f}(c,b)\cup \{0^{\infty},1^{\infty}\}\big)$ is countable.

\par Case 3, $H=(0,t)$ with $t<c$. Denote ${\rm \bf {t}}=\tau_{f}(t+)$ as the upper kneading sequence of $t$. By the definition of survivor set, we have
$$
\tilde{S}_{f}(0,t)=\{w\in\{0,1\}^{\mathbb{N}}: \sigma^{n}(w)=k(0) \ \textup{or} \  {\rm \bf {t}}\preceq\sigma^{n}(w) \ \forall n \in \mathbb{N}_{0}\}.
$$
 Indeed, except two sequences $0^\infty$ and $1^\infty$, $\tilde{S}_{f}(0,t)$ can be written as the union of two disjoint sets $\tilde{S}^{0}_{f}(0,t)$ and $\tilde{S}^{+}_{f}(0,t)$, where
\begin{align*}
\ \ \ \ \ \ \ \ \tilde{S}^{0}_{f}(0,t)&= \{w\in\{0,1\}^{\mathbb{N}}: \textup{there \ exists \ sequence} \ \{n_{i}\}_{i\geq1} \ \textup{ such \ that} \ \!\\
 & \sigma^{n_{i}}(w)=k(0) \ \textup{and} \ {\rm \bf {t}}\preceq\sigma^{n}(w)\ \textup{for \ all \ other } \ n\geq0  \},\!
\end{align*}
 and
$$\tilde{S}^{+}_{f}(0,t)= \{w\in\{0,1\}^{\mathbb{N}}: {\rm \bf {t}}\preceq\sigma^{n}(w)\preceq k(1) \ \forall n\geq0\}.
$$
It can be seen that $\tilde{S}^{0}_{f}(0,t)$ is a subset of $k(0)$'s preimages and hence is countable. As a result, $S_{f}(0,t)\setminus S^{+}_{f}(0,t)$ is countable.
\end{proof}
\end{lemma}

As we can see  from Lemma \ref{lemma2}, for all the holes $H$ we considered throughout this paper, the set $S_{f}(H)\setminus S^{+}_{f}(H)$ is at most countable. Moreover, the key point is that, the set $S^{+}_{f}(H)$ looks a lot like the Lorenz-shift, which can be used to calculate the topological entropy via kneading determinants in Section 4. Now we are in the position to prove Theorem \ref{xinthm1}, the notations used follow from the proof of Lemma \ref{lemma2}.

\vspace{0.4cm}
{\bf \noindent Proof of Theorem 1.1}
\par For simplicity, here we only give the proof of case $a=c$, and the case $b=c$ can be obtained similarly. Let $f\in ELM$, $(k_{+},k_{-})$ be the kneading invariants of $f$, $k(0)=\sigma(k_{+})$ and $k(1)=\sigma(k_{-})$. We divide the proof into two cases.
\par  Case 1, ${\rm \bf {b}}|_{2}=11$. It is clear that $f(b)\geq c$ since $\sigma({\rm \bf {b}})|_{1}=1$. By the proof of Case 3 in Lemma \ref{lemma2}, we have that $S_{f}(0,f(b))\subseteq \big(S^{0}_{f}(0,f(b))\cup\{0,1\}\big)$ and  $S_{f}(0,f(b))= \big(S^{0}_{f}(0,f(b))\cup\{0,1\}\big)$ if and only if both $0$ and $1$ are fixed points of $f$, where $S^{0}_{f}(0,f(b))$ is a subset of $0$'s preimages. Similar to the representation of $\tilde{S}^{0}_{f}(0,t)$ in Lemma \ref{lemma2},
\begin{align*}
\ \ \ \ \ \ \ \ \tilde{S}^{0}_{f}(0,f(b))&= \{w\in\{0,1\}^{\mathbb{N}}: \textup{there \ exists \ sequence} \ \{n_{i}\}_{i\geq1} \ \textup{ such \ that} \ \!\\
 & \sigma^{n_{i}}(w)=k(0) \ \textup{and} \ {\rm \bf {\sigma(b)}}\preceq\sigma^{n}(w)\ \textup{for \ all \ other } \ n\geq0  \}.\!
\end{align*}
 As for the hole $(c,b)$, by the proof of Case 2 in Lemma \ref{lemma2}, $\tilde{S}^{1}_{f}(c,b)$ is a subset of $k_+$'s preimages, and $\tilde{S}^{+}_{f}(c,b)\subseteq\{1^\infty\}$. Similarly, $S_{f}(b,c)\subseteq \big(S^{1}_{f}(b,c)\cup\{0,1\}\big)$ and  $S_{f}(c,b)= \big(S^{1}_{f}(c,b)\cup\{0,1\}\big)$ iff both $0$ and $1$ are fixed points.
 \par By the definition of $S_{f}(c,b)$ and $S_{f}(0,f(b))$, it is clear that $S_{f}(c,b)\supseteq S_{f}(0,f(b))$ for the reason that $f(c,b)=(0,f(b))$. As a result, we have $\tilde{S}^{1}_{f}(c,b)\supseteq \tilde{S}^{0}_{f}(0,f(b))$. We claim that, when $k(0)\in \tilde{S}^{0}_{f}(0,f(b))$, then $\tilde{S}^{1}_{f}(c,b)= \tilde{S}^{0}_{f}(0,f(b))$. Indeed, it is easy to check that $k(0)\in \tilde{S}^{0}_{f}(0,f(b))$ if and only if $\sigma^{n}(k(0))\succeq {\rm \bf {\sigma(b)}}$  for all $n\geq1$.
By the definition of $\tilde{S}^{1}_{f}(c,b)$, the inequality $\sigma^{n}(w)\preceq k_-$ holds naturally, hence we simplify the set $\tilde{S}^{1}_{f}(c,b)$ as
\begin{align*}
\ \ \ \ \ \ \ \ \tilde{S}^{1}_{f}(c,b)&= \{w\in\{0,1\}^{\mathbb{N}}: \textup{there \ exists \ sequence} \ \{n_{i}\}_{i\geq1} \ \textup{ such \ that} \ \sigma^{n_{i}}(w)=k_+,\!\\
 &   \textup{and} \ \sigma^{n}(w)\succeq{\rm \bf {b}} \ \textup{for \ all \ other } \ n\geq0  \ \textup{with} \ \sigma^{n}(w)|_{1}=1  \}.\!
\end{align*}
For any given $w\in \tilde{S}^{0}_{f}(0,f(b))$, there exists sequence $\{n_{i}\}_{i\geq1}$ such that $\sigma^{n_{i}}(w)=k(0)$, then we have $\sigma^{l_{i}}(w)=1k(0)=k_+$ for each $l_{i}=n_{i}-1$, and 1$\sigma^{n}(w)\succeq 1\sigma({\rm \bf {b}})={\rm \bf {b}}$ for other $n\geq0$. Hence $w\in \tilde{S}^{1}_{f}(c,b)$ and $\tilde{S}^{0}_{f}(0,f(b))\subseteq \tilde{S}^{1}_{f}(c,b)$. Conversely, given $w\in \tilde{S}^{1}_{f}(c,b)$, there exists sequence $\{n_{i}\}_{i\geq1}$ such that $\sigma^{n_{i}}(w)=k_+$, then we have $\sigma^{l_{i}}(w)=\sigma(k_+)=k(0)$ for each $l_{i}=n_{i}+1$, and $\sigma^{n+1}(w)\succeq \sigma({\rm \bf {b}})$ for all other $n\geq0$ with $\sigma^{n}(w)|_{1}=1$. Hence $w\in \tilde{S}^{0}_{f}(0,f(b))$ and $\tilde{S}^{0}_{f}(0,f(b))\supseteq\tilde{S}^{1}_{f}(c,b)$.
The proof of the claim is done. We give Example \ref{countableset} for an intuitive understanding of the two identical countable sets.

 \par  As for the case 2 ${\rm \bf {b}}|_{2}=10$, it can be proved similarly. Using the proof of Lemma \ref{lemma2}, it is clear that $\tilde{S}^{+}_{f}(0,f(b))=\tilde{S}^{+}_{f}(c,b)=\{ w \in \{ 0, 1\}^{\mathbb{N}} : \sigma({\rm \bf {b}}) \preceq  \sigma^{n}(\omega) \preceq k(1)\ \forall \, n \in \mathbb{N}_{0} \}$.  $\hfill\square$

\vspace{0.2cm}

 \begin{definition}\label{selfad}We say an infinite (or finite) sequence $\omega\in\{0,1\}^{\mathbb{N}}$ is {\bf self-admissible}, if $\omega$ is lexicographically smallest or lexicographically largest, that is, $\omega\preceq\sigma^{n}(\omega)$ holds for all $n\geq0$ or $\omega\succeq\sigma^{n}(\omega)$ holds for all $n\geq0$.
  \end{definition}

 It is clear that both $\sigma(k_{+})=k(0)$ and $\sigma(k_{-})=k(1)$ are self-admissible if $(k_{+},k_{-})$ are weak-admissible kneading invariants. Given infinite sequences $\mu,\nu\in\{0,1\}^{\mathbb{N}}$, similar to $\Omega(k_+,k_-)$, denote
$$\Omega(1\mu,0\nu)=\{w\in\{0,1\}^{\mathbb{N}}: \mu \preceq\sigma^{n}(w)\preceq \nu \ \forall n\geq0\}.$$

\begin{lemma}\label{le3}
Let $f\in ELM$ with a hole $H=(a,b)$. Then there exist self-admissible sequences $\mu$ and $\nu$ such that $\Omega(1\mu,0\nu)=\tilde{S}^{+}_{f}(a,b)$.
\end{lemma}
\begin{proof} Denote ${\rm \bf {a}}=(v_{0}v_{1}v_{2}\cdots)$ and ${\rm \bf {b}}=(u_{0}u_{1}u_{2}\cdots)$. By Lemma \ref{lemma2},
$ \tilde{S}^{+}_{f}(a,b)=\{ w \in \{ 0, 1\}^{\mathbb{N}} : \sigma({\rm \bf {b}}) \preceq  \sigma^{n}(\omega) \preceq \sigma({\rm \bf {a}})\ \forall \, n \in \mathbb{N}_{0} \}.$
 Without loss of generality, here we suppose $\sigma({\rm \bf {b}})$ is not self-admissible and show the existence of infinite sequence $\mu$.

 \par Since $\sigma({\rm \bf {b}})$ is not self-admissible, i.e., $\sigma({\rm \bf {b}})$ is not lexicographically smallest, there exists an integer $n$ such that $\sigma^{n}({\rm \bf {b}})\prec \sigma({\rm \bf {b}})$. Let $p=\inf\{n\geq 1: \sigma^{n}({\rm \bf {b}})\prec \sigma({\rm \bf {b}}) \}$, and $\mu=(u_{1}u_{2}\cdots u_{p})^{\infty}$. By the construction of $\mu$, it is clear that $\mu$ is self-admissible. Next we show that $\Omega(1\mu,0\nu)=\tilde{S}^{+}_{f}(a,b)$, where
 $$ \Omega(1\mu,0\nu)=\{ w \in \{ 0, 1\}^{\mathbb{N}} : (u_{1}u_{2}\cdots u_{p})^{\infty}\preceq  \sigma^{n}(\omega) \preceq \sigma({\rm \bf {a}})\ \forall \, n \in \mathbb{N}_{0} \}.
 $$
 Since $\sigma({\rm \bf {b}})\prec \mu$, it is clear that $\Omega(1\mu,0\nu)\subseteq\tilde{S}^{+}_{f}(a,b) $. For any $w=(w_{1}w_{2}\cdots)\notin \Omega(1\mu,0\nu)$, there exists an $n$ such that $\sigma^{n}(w)\prec \mu$ or $\sigma^{n}(w)\succ \sigma({\rm \bf {b}})$. If $\sigma^{n}(w)\succ \sigma({\rm \bf {b}})$, it is clear that $w\notin\tilde{S}^{+}_{f}(a,b)$. If $\sigma^{n}(w)\prec \mu$, then there exists an $l\geq0$ such that $w_{n+1}w_{n+2}\cdots w_{n+lp}=(u_{1}\cdots u_{p})^{l}$ and $w_{n+lp+1}\cdots w_{n+(l+1)p}\prec u_{1}\cdots u_{p} $, which indicates that $\sigma^{n+lp}(w)\prec \sigma({\rm \bf {b}})$ and hence $w\notin \tilde{S}^{+}_{f}(a,b)$. As a result, $\Omega(1\mu,0\nu)=\tilde{S}^{+}_{f}(a,b)$. The case $\sigma({\rm \bf {a}})$ not being self-admissible can be constructed similarly, and $\nu=(v_{1}v_{2}\cdots v_{q})^{\infty}$ where $q=\inf\{n\geq 1: \sigma^{n}({\rm \bf {a}})\succ\sigma({\rm \bf {a}}) \}$.
\end{proof}

As we know, if the pair $({\rm \bf {b}},{\rm \bf {a}})$ satisfies the weak-admissible condition in Section 2.2, i.e.,
	\begin{equation} \label{non-expandingcondition}
\sigma({\rm \bf {b}}) \preceq\sigma^{n}({\rm \bf {b}})\preceq\sigma({\rm \bf {a}}) \ \ {\rm and} \ \ \sigma({\rm \bf {b}})\preceq\sigma^{n}({\rm \bf {a}})\preceq\sigma({\rm \bf {a}}) \ \ \ \ {\rm for  \ all}\ \ n\geq 0,
	\end{equation}
then there exists a Lorenz map $g$ such that $\Omega(g)=\tilde{S}^{+}_{f}(a,b)$. However, for the case the pair $({\rm \bf {b}},{\rm \bf {a}})$ is not weak-admissible, can we still find a weak-admissible pair $(1s,0t)$ such that $\Omega(1s,0t)=\tilde{S}^{+}_{f}(a,b)$? See the following lemma.

\begin{lemma} \label{xin3}
Let $f\in ELM$ with hole $(a,b)$. If $\tilde{S}^{+}_{f}(a,b)\nsubseteqq\{0^\infty,1^\infty\}$ and $({\rm \bf {b}},{\rm \bf {a}})$ is not weak-admissible, then there exists weak-admissible $(1s,0t)$ such that $\Omega(1s,0t)=\tilde{S}^+_{f}(a,b)$.
\begin{proof} Denote ${\rm \bf {a}}=(v_{0}v_{1}v_{2}\cdots)$ and ${\rm \bf {b}}=(u_{0}u_{1}u_{2}\cdots)$. We know that $\Omega(1s,0t)=\{w\in\{0,1\}^{\mathbb{N}}: s \preceq\sigma^{n}(w)\preceq t \ \forall n\geq0\}$. Let $s=\min\tilde{S}^+_{f}(a,b)$ and $t=\max\tilde{S}^+_{f}(a,b)$. Since $\sigma({\rm \bf {b}})\preceq s\preceq t\preceq \sigma({\rm \bf {a}})$, it is clear that $\tilde{S}^{+}_{f}(a,b)\supseteq\Omega(1s,0t)$. Given $w\in \tilde{S}^{+}_{f}(a,b)$ and $n\geq 0$, we have $\sigma^{n}(w)\in \tilde{S}^{+}_{f}(a,b)$. Hence $s\preceq\sigma^{n}(w)\preceq t$ and $w\in \Omega(1s,0t)$. Next we show the construction of $s$ and $t$.
\par In fact, the condition $\tilde{S}^{+}_{f}(a,b)\nsubseteqq\{0^\infty,1^\infty\}$ makes sure that ${\rm \bf {b}}|_{2}=10$ and ${\rm \bf {a}}|_{2}=01$.
Since $({\rm \bf {b}},{\rm \bf {a}})$ is not weak-admissible, by Lemma \ref{le3}, we have $\tilde{S}^{+}_{f}(a,b)=\Omega(1s_{1},0t_{1})$, where both $s_{1}$ and $t_{1}$ are self-admissible.  If, in addition, $(1s_{1},0t_{1})$ is weak-admissible, i.e,
 	\begin{equation*}
s_{1} \preceq\sigma^{n}(1s_{1})\preceq t_{1} \ \ {\rm and} \ \ s_{1}\preceq\sigma^{n}(0t_{1})\preceq t_{1} \ \ \ \ {\rm for  \ all}\ \ n\geq 0,
	\end{equation*}
 then it is done. Otherwise, there exists an $n$ such that $\sigma^{n}(0t_{1})\prec s_{1}$ or $\sigma^{n}(1s_{1})\succ t_{1}$.
Since $\tilde{S}^{+}_{f}(a,b)\nsubseteqq\{0^\infty,1^\infty\}$, we consider the cases $\sigma^{n}(t_{1})\prec s_{1}$ or $\sigma^{n}(s_{1})\succ t_{1}$ $(n\geq1)$ for instead.
  \par Case 1, $\sigma^{n}(s_{1})\succ t_{1}$. Let $s_{2}=(u_{1}\cdots (u_{p}+1)))^\infty$ where $p=\inf\{n\geq 0: \sigma^{n}(s_{1})\succ t_{1}\}$. We claim that $\Omega(1s_{2},0t_{1})=\Omega(1s_{1},0t_{1})$. It is clear that $u_{p}=0$ and $s_2\succ s_1$, hence $\Omega(1s_{2},0t_{1})\subseteq\Omega(1s_{1},0t_{1})$. For any $w=(w_{1}w_{2}\cdots)\notin \Omega(1s_{2},0t_{1})$, there exists an $n$ such that $\sigma^{n}(w)\prec s_2$ or $\sigma^{n}(w)\succ t_1$. It is clear that $w\notin\Omega(1s_{1},0t_{1})$ if $\sigma^{n}(w)\succ t_1$. For the case $\sigma^{n}(w)\prec s_2$, then there exists an $l\geq0$ such that $w_{n+1}\cdots w_{n+lp}=(u_{1}\cdots u_{p-1}1)^{l}$ and $w_{n+lp+1}\cdots w_{n+(l+1)p}\prec u_{1}\cdots u_{p-1}1 $, which indicates that $w_{n+lp+1}\cdots w_{n+(l+1)p}\preceq u_{1}\cdots u_{p-1}0 $. Combining this with the assumption that $\sigma^{p}(s_{1})\succ t_{1}$, we have $\sigma^{n+lp}(w)\prec s_1$ and hence $w\notin\Omega(1s_{1},0t_{1})$. Repeat the process on $s_2$ if there exists $p_1<p$ such that $\sigma^{p_1}(s_2)\succ t_1$, the process will stop in finite times. Suppose it repeats for $m$ times, and $s_1\prec s_2 \prec\cdots \prec s_m$, then we can finally obtain $s=s_m=\min\tilde{S}^{+}_{f}(a,b)$.

  \par Case 2, $\sigma^{n}(t_{1})\prec s$. Let $t_{2}=(v_{1}\cdots (v_{q}-1)))^\infty$ where $q=\inf\{n\geq 0: \sigma^{n}(t_{1})\prec s\}$. Similarly, $\Omega(1s,0t_{2})=\Omega(1s,0t_{1})$, and the process will repeat for finite times. Suppose it repeats for $r$ times, and $t_1\succ t_2 \succ\cdots \succ t_r$, then $t=t_r=\max\tilde{S}^{+}_{f}(a,b)$.
 \par By the construction of $s$ and $t$, $(1s,0t)$ is weak-admissible and $\Omega(1s,0t)=\tilde{S}^+_{f}(a,b)$.
\end{proof}
\end{lemma}
We give Example \ref{consturcition of s,t} to understand the process of constructing weak-admissible $(1s,0t)$. Indeed, Lemma \ref{le3} and Lemma \ref{xin3} are the key to characterize the plateau of survivor set $\tilde{S}^+_{f}(a,b)$ in Section 4. Now we are in the position to prove Theorem 1.2, which will be used to calculate the topological entropy $h_{top} (S^+_{f}(H))$ via Lorenz shifts and kneading determinants.

\vspace{0.3cm}
\noindent {\bf Proof of Theorem \ref{theorem1}}
\par We start by proving the first part of the theorem. Let $f\in ELM$ with a hole $H=(a,b)$, if the pair $({\rm \bf {b}},{\rm \bf {a}})$ is weak-admissible, then it is clear that, there exists a Lorenz map $g$ such that $\Omega(g)=\tilde{S}^{+}_{f}(H)$. Moreover, the kneading invariants of $g$ is exactly $({\rm \bf {b}},{\rm \bf {a}})$. For the case the pair $({\rm \bf {b}},{\rm \bf {a}})$ is not weak-admissible, by Lemma \ref{xin3} above, we can still find a weak-admissible pair $(1s,0t)$. As a result, there exists $g\in LM$ with its kneading invariants being $(1s,0t)$, such that $\Omega(g)=\Omega(1s,0t)=\tilde{S}^{+}_{f}(H)$.
\par Conversely, for the second part, let $g\in LM$ and $(1s,0t)$ be its kneading invariants, then $\Omega(g)=\{w\in\{0,1\}^{\mathbb{N}}: s \preceq\sigma^{n}(w)\preceq t\ \forall n\geq0\}$. We claim that, for any $f\in ELM$ with $\Omega(f)\supsetneqq \Omega(g)$, there exists a hole $H=(a,b)$ which contains the critical point, such that $\Omega(g)=\tilde{S}^{+}_{f}(H)$. Denote $(k_+,k_-)$ as the kneading invariants of $f\in ELM$, since $\Omega(f)\supsetneqq \Omega(g)$, we have
$$  \sigma(k_+)\preceq s\prec0t\preceq k_-\prec k_+\preceq 1s\prec t\preceq \sigma(k_-) .
$$
As a result, we only need to choose $a$, $b$ such that $\tau_{f}(a-)=0t-$ and $\tau_{f}(b+)=1s$, and hence $\tilde{S}^{+}_{f}(a,b)=\Omega(g)$.      $\hfill\square$

\section{Plateau of $S^{+}_{f}(H)$ and devil staircase}\label{platformsec}
Let $f\in ELM$ with a hole $H=(a,b)$ and $a\leq c \leq b$. With the help of Lemma \ref{le3} and Lemma \ref{xin3}, we are now able to characterize that, when $a\leq c$ is fixed and $\rm \bf {b}$ is periodic, there exists a maximal interval $I(b)\subset(c,1)$ such that $S^{+}_{f}(a,\epsilon)$ is identical for all $\epsilon \in I(b)$, and $S^{+}_{f}(a,\eta)\neq S^{+}_{f}(a,b)$ for any $\eta \notin I(b)$. Moreover, $b$ is contained in $I(b)$ and we characterize the endpoints of $I(b)$ in different cases. The proof is quite different between the case $a=c$ and the case $a\neq c$, hence we prove the two cases separately.
\begin{proposition}\label{identicalinterval}
Let $f\in ELM$ with a hole $H=(c,b)$. If $\rm \bf {b}$ is periodic, then there exists a maximal interval $I(b)$ such that for all $\epsilon\in I(b)$, $S^{+}_{f}(c,\epsilon)=S^{+}_{f}(c,b)$. The endpoints of $I(b)$ are also characterized.
\begin{proof} Let $f\in ELM$, $(k_{+},k_{-})$ be its kneading invariants,
$k(0)=\sigma(k_{+})$ and $k(1)=\sigma(k_{-})$. Denote the period of $\rm \bf {b}$ as $p$ and $\sigma({\rm \bf {b}})=(v_{1}v_{2}\cdots v_{p})^{\infty}$. Notice that the case $H=(a,c)$ can be obtained similarly, hence we only give the proof of case $a=c$, which can be divided into the following cases.
\par Case 1, $({\rm {\bf {b}}},k_-)$ is weak-admissible. Given any $\epsilon\in[c,1]$ with its upper kneading sequence $\eta= v_{0}v_{1}v_{2}\cdots v_{p} \mu$, where $\mu$ is an infinite sequence satisfying $\mu\preceq(v_{1}v_{2}\cdots v_{p})^{\infty}$, it is clear that $\epsilon < b$ and we consider the following two sets

\begin{equation*}
\left \{ \begin{array}{ll}
\tilde{S}^{+}_{f}(c,b)=\{w\in\{0,1\}^{\mathbb{N}}: \sigma({\rm \bf {b}})\preceq\sigma^{n}(w)\preceq k(1) \ \forall n\geq0\},\\
\tilde{S}^{+}_{f}(c,\epsilon)=\{w\in\{0,1\}^{\mathbb{N}}: \sigma(\eta)\preceq\sigma^{n}(w)\preceq k(1) \ \forall n\geq0\}.
\end{array}
\right.
\end{equation*}
Applying the proof of Lemma \ref{le3}, we know that $\tilde{S}^{+}_{f}(c,b)=\tilde{S}^{+}_{f}(c,\epsilon)$. Hence the range of the interval $I(b)$ depends on the range of infinite sequence $\mu$. Indeed, in order to satisfy the weak-admissible condition, the lexicographically smallest $\mu$ can only be $k(0)$ and the lexicographically largest $\mu$ can only be $(v_{1}v_{2}\cdots v_{p})^{\infty}$. Denote the left and right endpoints of $I(b)$ are $b_{l}$ and $b_{r}$, we have $\tau_{f}(b_{l}+)=1v_{1}v_{2}\cdots v_{p} k(0)$
and $\tau_{f}(b_{r}+)={\rm \bf {b}}$, and $S^{+}_{f}(c,\epsilon)=S^{+}_{f}(c,b)$ for all $\epsilon\in[b_{l},b_{r}]$.

\par Case 2, $({\rm {\bf {b}}},k_-)$  is not weak-admissible, and there exists weak-admissible $(1s,k_-)$ such that $\Omega({\rm {\bf {b}}},k_-)=\Omega(1s,k_-)=\tilde{S}^+_{f}(c,b)$. Observe that $s$ is periodic for the reason that $\rm \bf {b}$ is periodic, and denote $s=(s_1\cdots s_r)^\infty$. Since $\sigma^{n}({\rm {\bf {b}}})$ is always lexicographically smaller than $\sigma(k_-)$ for any $n\geq 0$, $s$ can only be obtained for the reason that ${\rm {\bf {b}}}$ is not self-admissible. Similar to subcase 1, we can obtain that $I(b)=[b_{l},b_{r}]$, where $\tau_{f}(b_{l}+)=1s_1\cdots s_r k(0)$
and $\tau_{f}(b_{r}+)=1s$, and $S^{+}_{f}(c,\epsilon)$ is identical for all $\epsilon\in[b_{l},b_{r}]$.

\par Case 3, $({\rm {\bf {b}}},k_-)$  is not weak-admissible, and there exists weak-admissible $(1s,0t)$ such that $\Omega(1s,0t)=\Omega({\rm {\bf {b}}},k_-)=\tilde{S}^+_{f}(c,b)$. Notice that here $k_-$ is changed into sequence $0t$, which is different from Case 2. Denote $k(1)=\sigma(k_-)=(u_1 u_2 \cdots )$ and $s=(s_1\cdots s_r)^\infty$. If ${\rm {\bf {b}}}$ is unchanged, then let $s={\rm {\bf {b}}}$, and $k_-$ must be changed in Case 3. Suppose that $q$ is the smallest integer such that $\sigma^{q}(k(1))\prec s$ and $t=(u_1 \cdots u_{q-1} (u_{q}-1))^\infty$. Next we consider the sequence $\gamma=s_{1}s_{2}\cdots s_{r-1}(s_r -1)t$ in the following subcases.
\par Subcase 1, $\sigma^{q}(\sigma(k_-))   \preceq\gamma=s_{1}s_{2}\cdots s_{r-1}(s_r -1)t$. By the proof of Lemma \ref{xin3}, we have that $I(b)=(b_{l},b_{r}]$, where $\tau_{f}(b_{l}+)=1s_{1}s_{2}\cdots s_{r-1}0t$
and $\tau_{f}(b_{r}+)=1s$. Moreover, it can be easily seen that $1s_{1}s_{2}\cdots s_{r-1}0t\prec 1s_1\cdots s_{r-1}s_r k(0)$, hence the plateau in Subcase 1 contains the plateau in Case 2, and the left endpoint $b_{l}$ can not be reached.

\par Subcase 2, $\sigma^{q}(\sigma(k_-))   \succ\gamma=s_{1}s_{2}\cdots s_{r-1}(s_r -1)t$, and $u_{q+r}=1$. By the assumption in Case 3, we have that, $k_-$ is changed into $0t$ because $\sigma^{q}(\sigma(k_-))\prec s$. Hence
$$  s_{1}s_{2}\cdots s_{r-1}0t\prec\sigma^{q}(\sigma(k_-))\prec s=(s_1 s_2 \cdots s_{r-1} 1)^\infty.
$$
By Lemma \ref{xin3}, $k_-$ can not be changed into $0t$ no matter what $\gamma$ be. So the plateau here is the same as Case 2 above.

\par Subcase 3, $\sigma^{q}(k(1))   \succ\gamma=s_{1}s_{2}\cdots s_{r-1}(s_r -1)t$, $u_{q+r}=0$ and $\sigma^{q}(k(1))$ is self-admissible. In order to change $k_-$ into $0t$, we need that $t\succ \sigma^{q+r}(k(1))$. Hence $I(b)=(b_{l},b_{r}]$, where $\tau_{f}(b_{l}+)=1s_{1}s_{2}\cdots s_{r-1}\sigma^{q+r}(k(1))$
and $\tau_{f}(b_{r}+)=1s$.

\par Subcase 4, $\sigma^{q}(k(1))   \succ\gamma=s_{1}s_{2}\cdots s_{r-1}(s_r -1)t$, $u_{q+r}=0$ and $\sigma^{q}(k(1))$ is not self-admissible, which indicates that there exists $n>q$ such that $\sigma^{n}(k(1))\prec \sigma^{q}(k(1))$. At this case, we have
 $$ s_{1}s_{2}\cdots s_{r-1}(s_r -1)=u_{q+1}u_{q+2}\cdots u_{q+r-1}u_{q+r}.
 $$
 Since $\sigma^{q}(k(1))$ is not self-admissible, the sequence $s_{1}s_{2}\cdots s_{r-1}\sigma^{q+r}(k(1))$ in Subcase 3 is either not self-admissible. Applying the method in Case 2, we can obtain that $I(b)=[b_{l},b_{r}]$, where $\tau_{f}(b_{l}+)=1u_{q+1}\cdots u_{q+n}k(0)$
and $\tau_{f}(b_{r}+)=1s$.

\par See Example \ref{identical1} for an intuitive understanding of plateaux $I(b)$ in different cases.
\end{proof}
\end{proposition}
\begin{proposition}\label{identicalintervaldier}
Let $f\in ELM$ with a hole $H=(a,b)$, where $a<c<b$. If $\rm \bf {b}$ is periodic, then there exists a maximal interval $I(b)$ such that for all $\epsilon\in I(b)$, $S^{+}_{f}(a,\epsilon)$ is identical. The endpoints of $I(b)$ are also characterized.
\begin{proof} Here the plateau is quite different from Proposition \ref{identicalinterval}, for the reason that there exist lots of sequences between ${\rm {\bf {a}}}$ and $k_-$ when $a\neq c$. Let $f\in ELM$, $(k_{+},k_{-})$ be its kneading invariants,
$k(0)=\sigma(k_{+})$ and $k(1)=\sigma(k_{-})$. Denote the period of $\rm \bf {b}$ as $p$, $\sigma({\rm \bf {b}})=(v_{1}v_{2}\cdots v_{p})^{\infty}$, and $\sigma({\rm \bf {a}})=(u_{1}u_{2}\cdots)$. We divide the proof into following cases.
\par Case 1, $({\rm {\bf {b}}},{\rm {\bf {a}}})$ is weak-admissible, or $({\rm {\bf {b}}},{\rm {\bf {a}}})$ is not weak-admissible but only ${\rm {\bf {b}}}$ is modified. By Lemma \ref{xin3}, there exists weak-admissible $(1s,{\rm {\bf {a}}})$ such that $\Omega({\rm {\bf {b}}},{\rm {\bf {a}}})=\Omega(1s,{\rm {\bf {a}}})=\tilde{S}^+_{f}(a,b)$. Denote $1s={\rm {\bf {b}}}$ if $({\rm {\bf {b}}},{\rm {\bf {a}}})$ is weak-admissible; otherwise, we denote $s=(s_1\cdots s_r)^\infty$.
Similar to Case 3 in Proposition \ref{identicalinterval}, we can obtain that $I(b)=(b_{l},b_{r}]$, where $\tau_{f}(b_{l}+)=1s_{1}s_{2}\cdots s_{r-1}(s_r -1) \sigma({\rm {\bf {a}}})$
and $\tau_{f}(b_{r}+)=1s$.

\par Case 2, $({\rm {\bf {b}}},{\rm {\bf {a}}})$  is not weak-admissible and at least ${\rm {\bf {a}}}$ is modified.
 Then there exists weak-admissible $(1s,0t)$ such that $\Omega(1s,0t)=\Omega({\rm {\bf {b}}},{\rm {\bf {a}}}) =\tilde{S}^+_{f}(a,b)$. Denote $s=(s_1\cdots s_r)^\infty$, once $s$ is obtained, the sequence $t$ can be obtained in just one step. Let $j$ be the smallest integer such that $(u_1 u_2 \cdots u_j)^\infty$ is self-admissible, and $i$ be the smallest integer such that $\sigma^{i}(\sigma({\rm {\bf {a}}}))\prec s$. If such finite integers $j$ and $i$ does not exist, we denote them as $+\infty$. It is clear that $i\neq j$ since $\sigma^{j}(\sigma({\rm {\bf {a}}}))$ begins with $1$ and $\sigma^{i}(\sigma({\rm {\bf {a}}}))$ begins with $0$.
Now see the following subcases.
\par Subcase 1, $j<i$. Here we have that $t=(u_1 u_2 \cdots u_j)^\infty$. Similar to Case 3 in Proposition \ref{identicalinterval}, we have that $I(b)=(b_{l},b_{r}]$, where $\tau_{f}(b_{l}+)=1s_{1}s_{2}\cdots s_{r-1}(s_r -1) t$ and $\tau_{f}(b_{r}+)=1s$.
\par Subcase 2, $i<j$ and $\sigma^{i}(\sigma(k_-))\preceq\eta=s_{1}s_{2}\cdots s_{r-1}(s_r -1)t$, where $t=(u_1 \cdots u_{i-1} (u_{i}-1))^\infty$.
The plateau $I(b)=(b_{l},b_{r}]$ is the same form as Case 1, which means $\tau_{f}(b_{l}+)=1s_{1}s_{2}\cdots s_{r-1}(s_r -1) t=1\eta$ and $\tau_{f}(b_{r}+)=1s$.
\par Subcase 3, $i<j$, $\sigma^{i}(\sigma(k_-))   \succ\eta=s_{1}s_{2}\cdots s_{r-1}(s_r -1)t$ and $u_{i+r}=1$, where $t=(u_1 \cdots u_{i-1} (u_{i}-1))^\infty$. At this case, $s_{1}s_{2}\cdots s_{r-1}(s_r -1)\xi$ can not be changed into $s$ for all $t\preceq\xi\preceq \sigma({\rm {\bf {a}}})$. Hence the plateau $I(b)=(b_{l},b_{r}]$ is the same form as Subcase 1, which means $\tau_{f}(b_{l}+)=1s_{1}s_{2}\cdots s_{r-1}(s_r -1)\sigma({\rm {\bf {a}}})$ and $\tau_{f}(b_{r}+)=1s$.

\par Subcase 4, $i<j$, $\sigma^{i}(\sigma(k_-))\succ\eta=s_{1}s_{2}\cdots s_{r-1}(s_r -1)t$, $u_{i+r}=0$ and $\sigma^{i}(\sigma({\rm {\bf {a}}}))$ is self-admissible, where $t=(u_1 \cdots u_{i-1} (u_{i}-1))^\infty$. Similar to Subcase 3 in Proposition \ref{identicalinterval}, the plateau $I(b)=(b_l,b_r]$, where  $\tau_{f}(b_{l}+)=1s_{1}s_{2}\cdots s_{r-1}(s_r -1)\sigma^{i+r}(\sigma({\rm {\bf {a}}}))$ and $\tau_{f}(b_{r}+)=1s$.

\par Subcase 5, $i<j$, $\sigma^{i}(\sigma(k_-))\succ\eta=s_{1}s_{2}\cdots s_{r-1}(s_r -1)t$, $u_{i+r}=0$ and $\sigma^{i}(\sigma({\rm {\bf {a}}}))$ is not self-admissible, which indicates that there exists $n>i$ such that $\sigma^{n}(\sigma({\rm {\bf {a}}}))\prec \sigma^{i}(\sigma({\rm {\bf {a}}}))$. At this case, we have
 $$ s_{1}s_{2}\cdots s_{r-1}(s_r -1)\cdots s_n=u_{i+1}u_{i+2}\cdots u_{i+r-1}u_{i+r}\cdots u_{i+n}.
 $$
 Since $\sigma^{i}(\sigma({\rm {\bf {a}}}))$ is not self-admissible, the sequence $s_{1}s_{2}\cdots s_{r-1}\sigma^{q+r}(\sigma({\rm {\bf {a}}}))$ in Subcase 4 is either not self-admissible. Applying the method in Case 2 and Subcase 3, we can obtain that $I(b)=(b_{l},b_{r}]$, where $\tau_{f}(b_{l}+)=1u_{i+1}\cdots u_{i+n-1}(u_{i+n}-1)(\sigma({\rm {\bf {a}}}))$
and $\tau_{f}(b_{r}+)=1s$.
\end{proof}
\end{proposition}

\begin{remark}
\
\begin{enumerate}
\item When $a\neq c$, $I(b)$ is left open; when $a=c$, $I(b)$ maybe left closed.
\item Indeed, if both $a$ and $b$ can vary, there is a set $A\times B\subseteq I(a)\times I(b)$ such that for any $(\varepsilon,\eta)\in A\times B$, $\tilde{S}^{+}_{f}(\varepsilon,\eta)$ is identical. See Example \ref{identicaltwosets}.

\end{enumerate}
\end{remark}

It is known from above that there are lots of different plateaux on $[c,1]$, and the points that $S^{+}_{f}(a,b)$ changes are called bifurcation points. Let $f\in ELM$ with a hole $H=(a,b)$, now we are able to characterize the bifurcation set with $a$ being fixed, which defined as
$$ E_{f}(a):=\{b\in[c,1]:\tilde{S}^{+}_{f}(a,\epsilon)\neq \tilde{S}^{+}_{f}(a,b) \ {\rm for \ any} \ \epsilon>b\}.
$$
\begin{proposition}\label{pro2}
$E_{f}(a)=\{b\in[c,1]: {\rm \bf {b}}\in \tilde{S}^{+}_{f}(a,b)\}=\{b\in[c,1]:\sigma({\rm \bf {b}}) \preceq  \sigma^{n}({\rm \bf {b}}) \preceq \sigma({\rm \bf {a}}) \ \forall \ n\geq0\}$.
\begin{proof} Let $f\in ELM$ with $(k_+,k_-)$ being its kneading invariants, $k(1)=\sigma(k_-)$ and $k(0)=\sigma(k_+)$. Denote $A=\{b\in[c,1]:\sigma({\rm \bf {b}}) \preceq  \sigma^{n}({\rm \bf {b}}) \preceq \sigma({\rm \bf {a}}) \ \forall \ n\geq0\}$, $B=\{b\in[c,1]: {\rm \bf {b}}\in \tilde{S}^{+}_{f}(a,b)\}$ and $C=\{b\in[c,1]:\tilde{S}^{+}_{f}(a,\epsilon)\neq \tilde{S}^{+}_{f}(a,b) \ {\rm for \ any} \ \epsilon>b\}$. We claim that the bifurcation sets $A=B=C$ and prove this in two cases.
\par Case one, $a=c$. By Theorem \ref{xinthm1}, for any $b\in B$, we have ${\rm \bf {b}}\in \tilde{S}^{+}_{f}(c,b)=\tilde{S}^{+}_{f}(0,f(b))$, that is, $\sigma({\rm \bf {b}}) \preceq  \sigma^{n}({\rm \bf {b}}) \preceq k(1)$  for all $n\geq 0$ and hence $b\in A$. Conversely, given $b\notin B$, then there exists $N$ such that $k(0) \preceq  \sigma^{N}({\rm \bf {b}}) \prec \sigma({\rm \bf {b}})$, which indicates $b\notin A$. Hence we have $A=B$. Next we prove $B=C$. For any $b\in B$, we have ${\rm \bf {b}}\in \tilde{S}^{+}_{f}(0,f(b))$. Let $\epsilon>b$, by the monotonicity of $S^{+}_{f}(H)$ and $f$, $b\in (c,\epsilon)$, and hence $b\notin \tilde{S}^{+}_{f}(c,\epsilon)$. As a result, $\tilde{S}^{+}_{f}(c,b)\neq \tilde{S}^{+}_{f}(c,\epsilon)$ for all $\epsilon>b$, $b\in C$ and $B\subseteq C$. On the other side, given $b\notin B$, then there exits smallest integer $N$ such that $k(0) \preceq  \sigma^{N}({\rm \bf {b}}) \prec \sigma({\rm \bf {b}})$, that is, the first time enter the interval $[0,f(b))$. Actually, this relates to the definition of self-admissibility in Definition \ref{selfad}. Combining Lemma \ref{le3} with Proposition \ref{identicalinterval} above, we know that, there exists a plateau $I(b)=[b_{l},b_{r}]\ni b$, such that for any $\epsilon\in I(b)$, $\tilde{S}^{+}_{f}(c,b)=\tilde{S}^{+}_{f}(c,\epsilon)$. As a result, $b\notin B$ will lead to $b\notin C$. Moreover, the right endpoint $b_{r}$ of $I(b)$ belongs to $E_{f}(a)$ since its upper kneading sequence $(\sigma({\rm \bf {b}})|_{N})^\infty$ is admissible.
\par Case two, $a\neq c$.  For any $b\in B$, ${\rm \bf {b}}\in \tilde{S}^{+}_{f}(a,b)$ indicates that $\sigma({\rm \bf {b}}) \preceq  \sigma^{n}({\rm \bf {b}}) \preceq \sigma({\rm \bf {a}})$  for all $n\geq 0$, which leads to $b\in A$. Conversely, if $b\notin B$, then there exists $N$ such that $k(0) \preceq  \sigma^{N}({\rm \bf {b}}) \prec \sigma({\rm \bf {b}})$ or $\sigma^{N}({\rm \bf {b}}) \succ \sigma({\rm \bf {a}})$, which indicates $b\notin A$. Hence we have $A=B$. Next we prove $B=C$. For any $b\in B$, we have ${\rm \bf {b}}\in \tilde{S}^{+}_{f}(a,b)$. Let $\epsilon>b$, by the monotonicity of $S^{+}_{f}(H)$ and $f$, $b\in (a,\epsilon)$, and hence $b\notin \tilde{S}^{+}_{f}(a,\epsilon)$. As a result, $\tilde{S}^{+}_{f}(a,b)\neq \tilde{S}^{+}_{f}(a,\epsilon)$ for all $\epsilon>b$, $b\in C$ and $B\subseteq C$. On the other side, given $b\notin B$, then there exits smallest integer $N$ such that $k(0) \preceq  \sigma^{N}({\rm \bf {b}}) \prec \sigma({\rm \bf {b}})$ or $\sigma^{N}({\rm \bf {b}}) \succ \sigma({\rm \bf {a}})$, that is, the first time left the interval $(a,b)$. Similarly to Case 1 above,  this also relates with the definition of self-admissibility in Definition \ref{selfad}. Combining Lemma \ref{xin3} with Proposition \ref{identicalinterval} above, we know that, there exists a plateau $I(b)=(b^{\prime}_{l},b_{r}]\ni b$, such that for any $\epsilon\in I(b)$, $\tilde{S}^{+}_{f}(a,b)=\tilde{S}^{+}_{f}(a,\epsilon)$. As a result, $b\notin B$ will lead to $b\notin C$. Moreover, the right endpoint $b_{r}$ of $I(b)$ belongs to $E_{f}(a)$ since its upper kneading sequence is admissible.
\end{proof}
\end{proposition}

\begin{remark} \label{rem4}By Proposition \ref{pro2}, $b\in E_{f}(a)$ if and only if $\sigma({\rm \bf {b}})$ is admissible, that is, $\sigma({\rm \bf {b}}) \preceq  \sigma^{n}({\rm \bf {b}}) \preceq \sigma({\rm \bf {a}})$  for all $n\geq 0$, only self-admissibility is not enough. Moreover, if $b\notin E_{f}(a)$, and there exists $\epsilon>b$ with $\epsilon\in E_{f}(a)$ and $\tilde{S}^{+}_{f}(a,b)=\tilde{S}^{+}_{f}(a,\epsilon)$.
\end{remark}
\par Our aim is the topological entropy of survivor set $\tilde{S}_{f}(H)$. However, we always consider the set $\tilde{S}^{+}_{f}(H)$ for instead, for the reason that $S^{+}_{f}(H)$ is relative to the Lorenz shift whose entropy can be calculated via kneading determinants. Next we show that $\tilde{S}_{f}(H)$ and $\tilde{S}^{+}_{f}(H)$ have the same entropy. For a subset $\Omega\subseteq\{0,1\}^{\mathbb{N}}$, the topological entropy is defined as
$$
h_{top}(\sigma,\Omega):=\lim_{n\rightarrow \infty}\frac{\log(\#\Omega|n)}{n}=\inf_{n\rightarrow \infty}\frac{\log(\#\Omega|n)}{n},
$$
where the second equality holds for the reason that the sequence $\{\log(\#\Omega|n)\}$ is subadditive. We always write $h_{top}(\sigma,\Omega)$ as $h_{top}(\Omega)$ if we do not want to mention the shift map.

\begin{proposition}\label{pro1}
Let $f\in ELM$ with a hole $H=(a,b)\ni c$. Then $$h_{top}(\tilde{S}_{f}(H))=h_{top}(\tilde{S}^{+}_{f}(H)).$$
\begin{proof}
Denote $(k_{+},k_{-})$ be the kneading invariants of $f$ and $k(0)=\sigma(k_{+})$, $k(1)=\sigma(k_{-})$.
\par Case 1, $a<c<b$, by Lemma \ref{lemma2}, $\tilde{S}_{f}(a,b)\setminus\tilde{S}^{+}_{f}(a,b)\subseteq\{0^{\infty},1^{\infty}\}$. Using the definition of topological entropy, it is clear that $h_{top}(\tilde{S}_{f}(a,b))=h_{top}(\tilde{S}^{+}_{f}(a,b))$.
\par  Case 2, $a=c$, and the case $b=c$ can be proved similarly. Also by Lemma \ref{lemma2}, we know that $\tilde{S}_{f}(c,b)\setminus \tilde{S}^{+}_{f}(c,b)\subseteq \{0^{\infty},1^{\infty}\}\cup\tilde{S}^{1}_{f}(c,b)$ is a countable set. Recall that
\begin{align*}
\ \ \ \ \ \ \ \ \tilde{S}^{1}_{f}(c,b)&= \{w\in\{0,1\}^{\mathbb{N}}: \textup{there \ exists \ sequence} \ \{n_{i}\}_{i\geq1} \ \textup{ such \ that} \ \sigma^{n_{i}}(w)=k_+,\!\\
 &   \textup{and} \ \sigma^{n}(w)\succeq{\rm \bf {b}} \ \textup{for \ all \ other } \ n\geq0  \ \textup{with} \ \sigma^{n}(w)|_{1}=1  \}.\!
\end{align*}
Due to the monotonicity of entropy, we have $ h_{top}(\tilde{S}^{+}_{f}(c,b))\leq h_{top}(\tilde{S}_{f}(c,b))$ since $\tilde{S}^{+}_{f}(H)\subseteq \tilde{S}_{f}(H)$. Next we prove that $h_{top}(\tilde{S}^{+}_{f}(c,b))\geq h_{top}(\tilde{S}_{f}(c,b))$ in four subcases.
\par Case A, $a=b=c$, it is trivial that $\tilde{S}_{f}(c,b)=\tilde{S}^{+}_{f}(c,b)=\Omega(f)$.
\par Case B, $c\notin S_{f}(c,b)$. Since the set $S^{1}_{f}(c,b)$ is a subset of $c$'s preimages, it is clear that $S^{1}_{f}(c,b)=\emptyset$ when $c\notin S_{f}(c,b)$, and hence $h_{top}(\tilde{S}^{+}_{f}(c,b))= h_{top}(\tilde{S}_{f}(c,b))$ by Case 1.
\par Case C, $c\in S_{f}(c,b)$ and $b\notin E_{f}(a)$. By Remark \ref{rem4}, we know that there exists a plateau $I(b)$ such that for any $\epsilon\in I(b)$, $\tilde{S}^{+}_{f}(c,\epsilon)$ is identical. Moreover, the right endpoint of $I(b)$ satisfies $b_{r}\in E_{f}(a)$.
\par Therefore, it suffices to consider Case D, $b\in E_{f}(a)\setminus\{c\}$ and meanwhile $c\in S_{f}(c,b)$. Now we consider the set $\tilde{S}^{1}_{f}(c,b)$ defined above. Fix $m\in \mathbb{N}$, choose $\omega\in\tilde{S}^{1}_{f}(c,b)\setminus\{k_+\}$ and set $\zeta=\omega|_{m}$, then there exists the smallest positive integer $j$ such that $\sigma^{j}(\omega)=k_+$. Let $\nu=\omega|_{j-1}\varepsilon_{j}$, where $\varepsilon_{j}$ denotes the $j$-th symbol of $k_+$. Observe that
\begin{equation}\label{inequ}
{\rm \bf {b}}|_{j-i}\preceq\sigma^{i}(\nu)\preceq k(1)|_{j-i}
\end{equation}
holds for all $i\in\{0,1,\cdots,j-1\}$ with $\sigma^{i}(\nu)|_{1}=1$. Let $i^\ast$ be the smallest integer such that $\sigma^{i^\ast}(\nu)={\rm \bf {b}}|_{j-i^\ast}$, and set $i^\ast =j$ if the left inequality of (\ref{inequ}) holds strictly. By the minimality of integer $i^\ast$, since $b\in E_{f}(a)$ and $\nu|_{i^\ast}{\rm \bf {b}}$ satisfies $\sigma({\rm \bf {b}}) \preceq  \sigma^{n}(\nu|_{i^\ast}{\rm \bf {b}}) \preceq k(1)$  for all $n\geq 0$, we have $\nu\sigma^{j-i^\ast}({\rm \bf {b}})=\nu|_{i^\ast}{\rm \bf {b}}\in \tilde{S}^{+}_{f}(c,b)$. Moreover, $\nu|_{j-1} =\zeta|_{j-1}$ if $j\leq m$, and $\nu|_m =\zeta|_{m}$ if $j\geq m+1$. As a result, we have that for any $\zeta\in (\tilde{S}^{1}_{f}(c,b)\setminus\{k_+\})|_m$, $\zeta|_{j-1}\in\tilde{S}^{+}_{f}(c,b)|_{j-1}$ if $j\leq m$,  and $\zeta|_{m}\in\tilde{S}^{+}_{f}(c,b)|_{m}$ if $j\geq m+1$. Hence
$$ \#\tilde{S}^{1}_{f}(c,b)|_m\leq \sum_{j=1}^{m+1}\#\tilde{S}^{+}_{f}(c,b)|_{j-1}\leq(m+1)\#\tilde{S}^{+}_{f}(c,b)|_{m},
$$
for all $m\geq1$. Since $\tilde{S}_{f}(c,b)\setminus \tilde{S}^{+}_{f}(c,b)\subseteq \{0^{\infty},1^{\infty}\}\cup\tilde{S}^{1}_{f}(c,b)$, we have
$$ \#\tilde{S}_{f}(c,b)_m\leq 2+(m+2)\#\tilde{S}^{+}_{f}(c,b)|_{m}.
$$
Taking logarithms, dividing both side by $m$ and letting $m\rightarrow\infty$, we conclude that $h_{top}(\tilde{S}^{+}_{f}(c,b))\geq h_{top}(\tilde{S}_{f}(c,b))$, which gives the result.
\end{proof}
\end{proposition}

With the help of Proposition \ref{pro1}, we are able to calculate $h_{top}(\tilde{S}_{f}(a,b))$ via the entropy of $\tilde{S}^{+}_{f}(a,b)$. However, there exist lots of examples that topological entropy may remain the same  even the survivor set $\tilde{S}^{+}_{f}(a,b)$ changes. Hence we give the definition of  bifurcation set
$$ B_{f}(a):=\{b\in[c,1]:h_{top}(\tilde{S}^{+}_{f}(a,\epsilon))\neq h_{top}(\tilde{S}^{+}_{f}(a,b)) \ {\rm for \ any} \ \epsilon>b\},
$$
then naturally $B_{f}(a)\subseteq E_{f}(a)$. In order to obtain the devil staircase of entropy function, which is stated as Theorem \ref{theorem2}, we only need to prove that $E_{f}(a)$ is a Lebesgue null set.

\vspace{0.3cm}
\noindent {\bf Proof of Theorem \ref{theorem2}}
\par According to the assumption, $f\in ELM$ owns an ergodic acim and $H=(a,b)\ni c$. The ergodicity of $f$ with respect to its invariant measure equivalent to the Lebesgue measure $\mu$ implies that $\mu$-almost every $x\in[0,1]$ is eventually mapped into the interval $(a,c+\frac{1-c}{N})$ for any $N\in\mathbb{N}$, no matter $a=c$ or not. As a result, the survivor sets $S^{+}_{f}(a,b)$ and $S_{f}(a,b)$ are of zero Lebesgue measure. We claim that
$$E_{f}(a)\subseteq \bigcup_{N=1}^{\infty}S^{+}_{f}\bigg(a,c+\frac{1-c}{N}\bigg).$$
By Proposition \ref{pro2}, $E_{f}(a)=\{b\in[c,1]: {\rm \bf {b}}\in \tilde{S}^{+}_{f}(a,b)\}$, and hence $E_{f}(a)\cap [f(b),f(a)]\subseteq S^{+}_{f}(a,b)$ for any $b\in[c,1]$. Let $b=c+\frac{1-c}{N}$, then we have

\begin{align*}
\ \ \ \ \ \ \ \ &\bigcup_{N=1}^{\infty}\bigg(E_{f}(a)\cap\big[f(c+\frac{1-c}{N}),f(a)\big]\bigg)\!\\ &= E_{f}(a)\cap\bigcup_{N=1}^{\infty}\bigg(\big[f(c+\frac{1-c}{N}),f(a)\big]\bigg)\!\\
 &= E_{f}(a)\cap [0,f(a)] \!\\
  &= E_{f}(a)\subseteq  \bigcup_{N=1}^{\infty}S^{+}_{f}\bigg(a,c+\frac{1-c}{N}\bigg). \!
\end{align*}
By subadditivity of the Lebesgue measure, it follows that $E_{f}(a)$ is a Lebesgue null set. Combining this with Proposition \ref{continuity}, we obtain that the entropy function $\lambda_{f}(a):b\mapsto h_{top}(\tilde{S}_{f}(H))$ is a devil staircase, where $H=(a,b)$ and $a$ is fixed. More precisely, $\lambda_{f}(a)$ is decreasing and constant for Lebesgue almost everywhere.          $\hfill\square$

\section{Final comments}\label{comments}

 \par Although we extend the devil staircases in \cite{Urbanski1987,kalle2020,Langeveld2023} to $f\in ELM$ with a hole at critical point $c$, it is still unknown how to deal with the cases that $H\subsetneqq(0,c)$ and $H\subsetneqq(c,1)$, because the tools from Lorenz shift can not be used directly. We only know that, at the special case that the preimage of $(a,b)$ ($a\leq c \leq b$) is unique and $(d,e)=f^{-1}(a,b)$ is the proper subset of $(0,c)$ or $(c,1)$, $S_{f}(a,b)=S_{f}(d,e)$.
\par Indeed, the devil staircase of entropy function is proved via the bifurcation set $E_{f}(a)$, for the reason that $B_{f}(a)\subseteq E_{f}(a)$. A natural question arises: at which case the two bifurcation sets coincide?  When $f$ being $T_\beta$ ($\beta\in(1,2]$) with a hole $(0,t)$, it was proved by Baker and Kong \cite{baker2020} that, two bifurcation sets coincide at the case $\beta$ being multinacci numbers. Hence, at the case $f\in ELM$ and $c\in H$, can we give a sufficient and necessary condition for the equivalence of sets $B_{f}(a)$ and $ E_{f}(a)$?

\section*{Appendix}\label{Some examples}

\begin{example}\label{countableset} (Characterization of $\tilde{S}^{0}_{f}(0,f(b))$ and $\tilde{S}^{1}_{f}(c,b)$ in the proof of Theorem \ref{xinthm1})
\
\par Let $f\in ELM$ with its kneading invariants being $(k_+,k_-)=((100)^\infty,01^\infty)$. Here we consider two kinds of holes $H_{1}=(c,b)$ and $H_{2}=(0,f(b))$ with different parameters $b$, which leads to the difference between sets $\tilde{S}^{0}_{f}(0,f(b))$ and $\tilde{S}^{1}_{f}(c,b)$.
\begin{enumerate}
\item  Let ${\rm \bf {b_1}}=(10)^\infty$ be the upper kneading sequence of $b_1$. It can be seen that $\sigma({\rm \bf {b_1}})=(01)^\infty$ and $1$ is the fixed point of $f$. By the proof of Theorem \ref{xinthm1}, $k(0)\in\tilde{S}^{0}_{f}(0,f(b_1))$ if and only if $\sigma^{n}(k(0))\succeq \sigma({\rm \bf {b_1}})$ for all $n\geq1$, hence $\tilde{S}^{0}_{f}(0,f(b_1))=\emptyset$ for $\sigma^{1}(k(0))\prec (01)^\infty$. As for the set $\tilde{S}^{1}_{f}(c,b_1)$, it can be checked that $k_{+}\in\tilde{S}^{1}_{f}(c,b_1)$ and hence a countable subset of $k_+$'s preimages. Hence $S_{f}(c,b_1)\setminus S_{f}(0,f(b_1))$ is countable.

\item  Let ${\rm \bf {b_2}}=(10010)^\infty$ be the upper kneading sequence of $b_2$ and $\sigma({\rm \bf {b_2}})=(00101)^\infty$. It can be checked that $k(0)\in\tilde{S}^{0}_{f}(0,f(b_2))$. According to the proof of Theorem \ref{xinthm1}, we obtain $\tilde{S}^{0}_{f}(0,f(b_2))=\tilde{S}^{1}_{f}(c,b_2)$ and $S_{f}(0,f(b_2))=S_{f}(c,b_2)$.
\end{enumerate}

\end{example}

\begin{example}\label{consturcition of s,t}(Construction of weak-admissible $(1s,0t)$)
\
\par For simplicity, here we consider $f=T_{2}$, hence $\Omega(f)$ is the full shift $\{0,1\}^\mathbb{N}$. Let $H=(a,b)$ be the hole, denote $\sigma({\rm \bf {b}})=(0101111010)^\infty$ and $\sigma({\rm \bf {a}})=(111001011110)^\infty$. Next we construct weak-admissible $(1s,0t)$ such that $\Omega(1s,0t)=\tilde{S}^{+}_{f}(a,b)$. By Lemma \ref{le3}, we obtain self-admissible sequences $s_{1}=(0101111)^\infty$ and $t_{1}=(1110010)^\infty$ such that $\Omega(1s_1,0t_1)=\tilde{S}^{+}_{f}(a,b)$. Using Lemma \ref{xin3}, we have $s=s_{2}=(011)^\infty$ since $\sigma^{3}(s_{1})\succ t_1$. Observe that $\sigma^{3}(t_{1})\prec s_2$, then we have $t=t_{2}=(110)^\infty$. As a result, $\tilde{S}^{+}_{f}(a,b)=\Omega(1s,0t)=\Omega((101)^\infty,(011)^\infty)=\{(101)^\infty,(011)^\infty,(110)^\infty\}$.

\end{example}

\begin{example}\label{identical1}(Plateau of case $(c,b)$)
\
\par   Let $f\in ELM$ with a hole $(a,b)$ and its kneading invariants being $(k_+,k_-)$. Denote $k(1)=\sigma(k_-)$ and  $k(0)=\sigma(k_+)$. There are six cases in the proof of Proposition \ref{identicalinterval}, here we show the differences of interval $I(b)$.
\begin{enumerate}
\item  Let $(k_+,k_-)=(10^\infty,(011001)^\infty)$, $H_{1}=(c,b_1)$ with ${\rm \bf {b_1}}=(100)^\infty$ and $H_{2}=(c,b_2)$ with ${\rm \bf {b_2}}=(1001000)^\infty$.  It can be seen that $\Omega({\rm \bf {b_1}},k_-)$ is weak-admissible. By Case 1 in Proposition \ref{identicalinterval}, we have that $I(b_1)=[b_{l},b_r]$, where $\tau_{f}(b_{l}+)=1001k(0)=100k_+=10010^\infty$ and $\tau_{f}(b_{r}+)={\rm \bf {b_1}}=(100)^\infty$. For each $\epsilon \in I(b_1)$, $\tilde{S}^{+}_{f}(c,\epsilon)=\Omega((100)^\infty,(011001)^\infty)$.
And $I(b_1)$ is the maximal interval such that the survivor set $\tilde{S}^+_{f}(c,b_1)$ is identical. As for the hole $(c,b_2)$, $\Omega({\rm \bf {b_2}},k_-)$ is not self-admissible and only ${\rm \bf {b_2}}$ need to be modified. By Case 2 in Proposition \ref{identicalinterval}, we have that $I(b_2)=I(b_1)$.

\item  Let $(k_+,k_-)=(10^\infty,(0110010)^\infty)$ and $H_{3}=(c,b_{3})$ with ${\rm \bf {b_3}}=(10)^\infty$.  It can be seen that $\Omega({\rm \bf {b_3}},k_-)$ is not weak-admissible, by Lemma \ref{xin3}, there exists weak admissible $(1s,0t)$ such that $\Omega({\rm \bf {b_3}},k_-)=\Omega(1s,0t)$, where $s=(01)^\infty$ and $t=(10)^\infty$. Next we consider the sequence $\eta=00(10)^\infty$. According to Subcase 1 in Proposition \ref{identicalinterval}, $\eta\succeq \sigma^{2}(\sigma(k_-))$, hence $I(b_3)=(b_{l},b_r]$, where $\tau_{f}(b_{l}+)=100t=10(01)^\infty$ and $\tau_{f}(b_{r}+)=(10)^\infty$. Moreover, $I(b_3)$ is the maximal interval such that the survivor set $\tilde{S}^+_{f}(c,b_3)$ is identical, and the left endpoint $b_{l}$ can not be reached.

\item  Let $(k_+,k_-)=(10^\infty,(0111001000)^\infty)$ and $H_{4}=(c,b_{4})$ with ${\rm \bf {b_4}}=(100)^\infty$.  It can be seen that $\Omega({\rm \bf {b_4}},k_-)$ is not weak-admissible, and we can find $s=(001)^\infty$ and $t=(110)^\infty$ such that $\Omega({\rm \bf {b_4}},k_-)=\Omega(1s,0t)$. Denote $\eta=000(110)^\infty$ and we have
    $$\eta=000(110)^\infty\prec \sigma^{3}(\sigma(k_-))=(0010000111)^\infty)\prec s=(001)^\infty.$$
    By Subcase 2 in Proposition \ref{identicalinterval}, $I(b_4)=(b_{l},b_r]$, where $\tau_{f}(b_{l}+)=1001k(0)=10010^\infty$ and $\tau_{f}(b_{r}+)=(100)^\infty$.

\item  Let $(k_+,k_-)=(10^\infty,(01110010011)^\infty)$ and $H_{5}=(c,b_{5})$ with ${\rm \bf {b_5}}=(10010)^\infty$.  It can be seen that $\Omega({\rm \bf {b_5}},k_-)$ is not weak-admissible, and we can find $s=(00101)^\infty$ and $t=(110)^\infty$ such that $\Omega({\rm \bf {b_5}},k_-)=\Omega(1s,0t)$. Denote $\eta=00100(110)^\infty$ and we have
    $$\eta=00100(110)^\infty\prec \sigma^{3}(\sigma(k_-))=(00100110111)^\infty)\prec s=(00101)^\infty.$$
    Moreover, $\sigma^{3}(\sigma(k_-))$ is self-admissible.
    By Subcase 3 in Proposition \ref{identicalinterval}, $I(b_5)=(b_{l},b_r]$, where $\tau_{f}(b_{l}+)=1\sigma^{3}(\sigma(k_-))$ and $\tau_{f}(b_{r}+)=(10010)^\infty$.

\item  Let $(k_+,k_-)=(10^\infty,(0111000111000)^\infty)$ and $H_{6}=(c,b_{6})$ with ${\rm \bf {b_6}}=(100)^\infty$.  It can be seen that $\Omega({\rm \bf {b_6}},k_-)$ is not weak-admissible, and we can find $s=(001)^\infty$ and $t=(110)^\infty$ such that $\Omega({\rm \bf {b_6}},k_-)=\Omega(1s,0t)$. Denote $\eta=000(110)^\infty$ and we have
    $$\eta=000(110)^\infty\prec \sigma^{3}(\sigma(k_-))=(0001110000111)^\infty)\prec s=(001)^\infty.$$
    Moreover, $\sigma^{3}(\sigma(k_-))$ is not self-admissible.
    By Subcase 4 in Proposition \ref{identicalinterval}, $I(b_6)=[b_{l},b_r]$, where $\tau_{f}(b_{l}+)=1000111k(0)$ and $\tau_{f}(b_{r}+)=(100)^\infty$.
\end{enumerate}
\end{example}

\begin{example}\label{identicaltwosets}(The set $I(a)\times I(b)$)
\
\par Here we consider $f=T_{2}$, hence $c=1/2$ and $\Omega(f)$ is the full shift $\{0,1\}^\mathbb{N}$. Let $H=(a,b)\ni c$ be the hole, denote $\sigma({\rm \bf {b}})=(01)^\infty$ and $\sigma({\rm \bf {a}})=(110)^\infty$. We calculate $I(a)$ and $I(b)$ separately. By the proof of Proposition \ref{identicalintervaldier}, we obtain that $I(b)=(b_{l},b_{r}]$, where $\tau_{f}(b_{l}+)=100\sigma({\rm \bf {a}})=100(110)^\infty=10(011)^\infty$ and $\tau_{f}(b_{r}+)={\rm \bf {b}}=(10)^\infty$.
\par As for the interval $I(a)$, similarly, we also obtain that $I(a)=[a_{l},a_{r})$, where $\tau_{f}(a_{l}-)={\rm \bf {a}}=(011)^\infty$ and $\tau_{f}(a_{r}-)=011\sigma({\rm \bf {b}})=011(10)^\infty$. Moreover, we denote $\tau_{f}(b^\prime_{l}+)=1010^\infty$ and $\tau_{f}(a^\prime_{r}-)=01101^\infty$.
Let $A=(b_l,b_r]$ and $B=[a_l,a^\prime_r]$, we can see that, for any $(\varepsilon,\eta)\in A\times B$, we have
$$\tilde{S}^{+}_{f}(\varepsilon,\eta)=\Omega((10)^\infty,(011)^\infty),$$
where $\Omega((10)^\infty,(011)^\infty)=\{ w \in \{ 0, 1\}^{\mathbb{N}} : (01)^\infty\preceq  \sigma^{n}(\omega) \preceq (110)^\infty\ \forall \, n \in \mathbb{N}_{0} \}.$ Similarly, it can be verified that the set $[b^\prime_l,b_r]\times[a_l,a_r)$ is also a plateau set. However, we can not extend to plateau set to $I(a)\times I(b)$ since we have a counterexample in the set $(b^\prime_l,b_l)\times(a^\prime_r,a_r)$. Let $\kappa \in (b^\prime_l,b_l)$ and $\lambda \in(a^\prime_r,a_r)$ with $\tau_{f}(\kappa_+)=(10011)^\infty$ and $\tau_{f}(\lambda_-)=(01110)^\infty$, then the survivor set $\tilde{S}^{+}_{f}(\kappa,\lambda)\neq\tilde{S}^{+}_{f}(a,b)$.

\end{example}

\end{document}